\documentclass[a4paper,11pt]{article}

\usepackage{amsthm}
\usepackage{amssymb}
\usepackage{amsmath}

\usepackage{mathtools}
\usepackage{enumitem}
\usepackage{multirow}
\usepackage[normalem]{ulem}

\usepackage{fancyhdr}
\usepackage[titles]{tocloft}
\usepackage{pgf, tikz, pgfplots}
\pgfplotsset{compat=1.17}
\usetikzlibrary{calc, arrows, shapes, positioning, patterns, angles, quotes, intersections, through, backgrounds}
\usetikzlibrary{decorations.markings}

\usepackage{hyperref}

\newtheorem{thm}{Theorem}[section]
\newtheorem*{thm*}{Theorem}
\newtheorem{pop}[thm]{Proposition}
\newtheorem{lem}[thm]{Lemma}

\newtheorem{cor}[thm]{Corollary}
\newtheorem{defi}[thm]{Definition}

\newtheorem{manualtheoreminner}{Theorem}
\newenvironment{introtheorem}[1]{%
  \renewcommand\themanualtheoreminner{#1}%
  \manualtheoreminner
}{\endmanualtheoreminner}

\theoremstyle{definition}
\newtheorem{rem}[thm]{Remark}
\newtheorem{ex}[thm]{Example}
\numberwithin{equation}{section}

\AtBeginEnvironment{ex}{%
  \pushQED{\qed}%
}
\AtEndEnvironment{ex}{\popQED\endex}

\AtBeginEnvironment{rem}{%
  \pushQED{\qed}%
}
\AtEndEnvironment{rem}{\popQED\endex}

\theoremstyle{definition} 

\newcommand{\thistheoremnam}{}
\newtheorem*{genericthm*}{\thistheoremnam}
\newenvironment{chapt*}[1]
  {\renewcommand{\thistheoremnam}{#1}%
   \begin{genericthm*}}
  {\end{genericthm*}}

\newcommand{\todo}[1]{$\clubsuit$ {\tt #1} $\clubsuit$}
\newcommand{\LLS}{Lo\-rentz\-ian length space }

\newcommand{\LpLS}{Lo\-rentz\-ian pre-length space }
\newcommand{\LpLSn}{Lo\-rentz\-ian pre-length space}
\newcommand{\LpLSs}{Lo\-rentz\-ian pre-length spaces }
\newcommand{\LpLSsn}{Lo\-rentz\-ian pre-length spaces}

\newcommand{\mb}[1]{\mathbb{#1}}

\newcommand{\widebar}[1]{\overline{#1}}
\usepackage{accents}

\DeclareMathOperator{\diam}{diam}
\DeclareMathOperator{\arcosh}{arcosh}

\newcommand{\diamfin}{\ensuremath{\diam_{\mathrm{fin}}}}
\newcommand{\lm}[1]{\mathbb{L}^2(#1)}
\newcommand{\ma}{\ensuremath{\measuredangle}}

\newcommand{\lpls}{(X,d,\ll,\leq,\tau)}

\newcommand{\bx}{\bar{x}}
\newcommand{\by}{\bar{y}}
\newcommand{\bz}{\bar{z}}
\newcommand{\bp}{\bar{p}}
\newcommand{\bq}{\bar{q}}

\newcommand{\N}{\mathbb{N}}

\providecommand\given{} 
\newcommand\SetSymbol[1][]{
   \nonscript\,#1\vert \allowbreak \nonscript\,\mathopen{}}
\DeclarePairedDelimiterX\Set[1]{\lbrace}{\rbrace}%
 { \renewcommand\given{\SetSymbol[\delimsize]} #1 }

 \newcommand{\newfoot}[1]{\footnotemark\ \hspace{-2mm}\footnotetext{#1}}

 \fancypagestyle{firstpage}
{
    \fancyhead[L]{}    
    \fancyhead[R]{DMUS-MP-23/01}
}

\newenvironment{acknowledgements}{%
\begin{abstract}
}{%
\end{abstract}
}

\newcommand{\tb}[1]{{\color{purple}{#1}}}

\renewcommand{\labelenumi}{(\roman{enumi})}
\renewcommand\theenumi\labelenumi

\usepackage[numbers,sort]{natbib}
\title{ Alexandrov's Patchwork and the Bonnet--Myers Theorem for Lorentzian length spaces}
\date{}
\author{Tobias Beran\thanks{{\tt tobias.beran@univie.ac.at}, Faculty of Mathematics, University of Vienna, Austria.}\ , Lewis Napper\thanks{{\tt lewis.napper@surrey.ac.uk}, Department of Mathematics, University of Surrey, UK.}\ \ and Felix Rott\thanks{{\tt felix.rott@univie.ac.at}, Faculty of Mathematics, University of Vienna, Austria.}}

\begin{document}
\maketitle
\thispagestyle{firstpage}

\begin{abstract}
We present several key results for \LpLSs with global timelike curvature bounds. Most significantly, we construct a Lorentzian analogue to Alexandrov's Patchwork, thus proving that suitably nice \LpLSs with local upper timelike curvature bound also satisfy a corresponding global upper bound. Additionally, for spaces with global lower bound on their timelike curvature, we provide a Bonnet--Myers style result, constraining their finite diameter.
Throughout, we make the natural comparisons to the metric case, concluding with a discussion of potential applications and ongoing work.
\vspace{1em}

\medskip

\end{abstract}
\medskip

\begin{acknowledgements}
We want to thank James Grant and Argam Ohanyan for helpful comments throughout this project. We also want to acknowledge the kind hospitality of the Erwin Schrödinger International Institute for Mathematics and Physics (ESI) during the workshop \emph{Nonregular Spacetime Geometry}, where some of this research was carried out as well as the Fields Institute of the University of Toronto, which kindly supported the authors' attendance at the \emph{Workshop on Mathematical Relativity, Scalar Curvature and Synthetic Lorentzian Geometry} in October 2022. 
Finally, we would like to thank a number of reviewers for their valuable comments on the working version of this paper.

This work was supported by research grant P33594 of the Austrian Science Fund FWF.
\end{acknowledgements}

\newpage

\tableofcontents

\section{Introduction}

By utilising the theory of metric length spaces, the scope of many results in differential geometry can be extended beyond the setting of smooth manifolds. In particular, metric length spaces are a key tool in the abstraction of the fundamental properties of Riemannian manifolds, to structures of lower regularity, see the books by \citet*{BBI01} and \citet*{BH99}. In this so-called synthetic approach, curvature bounds (locally/in the small) are constructed via the comparison properties of geodesic triangles and are used to tame some of the more pathological behaviour of such length spaces. A semi-Riemannian extension of these comparison methods has also been developed by \citet*{AB08} and \citet*{Har82}. 
\bigskip

The properties that arise from supplementing length spaces with global curvature bounds (producing spaces of Alexandrov or CAT($k$) type) have also been significant for the application of metric length spaces to problems in a variety of fields, from dynamical systems to group theory. An illustrative result by \citet*{BGP92}, in the case of curvature bounded below, is the stability of global lower curvature bounds on metric spaces under Gromov--Hausdorff limits (see \citet*{Kap02} for a more detailed discussion). We can see the impact of spaces with global upper curvature bounds by looking to algebraic topology: the fundamental group of any complete metric space with curvature bounded above by zero has no non-trivial finite subgroups, see \citet*[Corollary 9.3.2]{BBI01}. 
\bigskip

As such, a salient question asks under which conditions can curvature bounds, which are imposed locally, be extended to hold in the large? For spaces with curvature bounded above, \citet*{Ale57} demonstrated that, provided we also have unique geodesics which vary continuously with their endpoints, local CAT($k$) spaces are CAT($k$). Extending the work of \citet*{Top59} on Riemannian manifolds, it was shown by \citet*{BGP92} that for curvature bounded below, complete length spaces are sufficient. In fact, a generalization of the Bonnet--Myers Theorem (cf.\ \citet*{BBI01}) follows as a natural corollary of the aforementioned result, bounding the diameter of complete length spaces with local curvature bounded below. A further globalization result, which treats completions of geodesic spaces with curvature bounded below, has been obtained by \citet*{Pet16}. 
\bigskip

Analysis of low regularity Lorentzian geometry has become increasingly pertinent in the study of general relativity and physically relevant space-times with singularities, for example, it is known that the vacuum Einstein equations admit solutions of Sobolev regularity $H^s_{loc}$, provided $s > \tfrac{5}{2}$ and the induced Riemannian metric on spacelike slices is also $H^s_{loc}$, see \citet*{Ren05}. In the case where the Lorentzian metric is $C^1$, it has recently been shown by \citet*{Gra20} that the Hawking and Penrose singularity theorems persist and physically reasonable Lorentzian manifolds with such a metric cannot be causal geodesically complete. For examples of physically relevant models with even lower regularity, we direct the reader to the class of impulsive gravitational waves, which are described by either a continuous or distributional metric and admit unique, continuously differentiable geodesics, see Podolsk{\'y}, S{\"a}mann, Steinbauer, and \u{S}varc \cite{PSSS15, PSSS16}. Alternatively, straight conical cosmic strings possess distributional energy momentum tensors and may be described by totally geodesic, quasi-regular singularities, with the text by \citet*{Vic90} discussing the implications of this low regularity structure on the dynamics of more general cosmic strings. Furthermore, \citet*{Mons23} have recently formulated the Einstein equations in terms of optimal transport and provide an example of bounding the Ricci curvature of FLRW spaces with warping functions in $H^1$. This has been extended to general synthetic Lorentzian spaces by \citet*{CM20}, with \citet*{McC20} providing a similar formulation of the strong energy condition for globally hyperbolic spacetimes.
\bigskip

Hence, the introduction of the \LpLS by \citet*{KS18} as a Lorentzian analogue to the metric space and an extension of the causal spaces of \citet*{KP67}, has led to the rapid development of a synthetic Lorentzian framework, with the vast number of novel results mirroring the growth of the theory of metric length spaces several decades ago. See for example the work of Alexander, Graf, Grant, Kunzinger, S{\"a}mann, and Steinbauer \cite{AGKS19,GKS19,KS21}. However, while the development of metric length spaces was guided by disciplines such as group theory and the study of partial differential equations alongside its purely geometric origin, research into \LpLSs has, for the most part, focused on its apparent necessity in general relativity. This paper, continuing from the work of Beran, Ohanyan, Rott, S{\"a}mann, and Solis \cite{BR22,Rot22,BS22,BORS23}, aims to develop suitable Lorentzian analogues to some of the fundamental results of metric length spaces, in order to facilitate the application of the \LLS framework to a wider range of disciplines. In particular, given their myriad of applications in the metric setting, we focus on the notion of spaces with synthetic curvature bounds and their properties. Curvature bounds on Lorentzian pre-length spaces were first described by \citet*{KS18} via the means of triangle comparison, with alternative approaches using hyperbolic angles being proposed by  \citet*{BS22} and \citet*{BMS22} in recent years. However, these works do not attempt to relate curvature bounds enforced globally to those imposed on neighbourhoods --- this gap in the literature is one which this paper aims to fill, in the case of curvature bounded above.
\bigskip

An outline of the paper is as follows. We begin in Section \ref{sec:prelim} by re-iterating some basic definitions regarding \LpLSsn, $\tau$-length, and the causal ladder. We also introduce the notion of a regular \LpLSn, the technique of triangle comparison, and both local and global timelike curvature bounds. Similar definitions in the context of metric spaces are also provided for convenience and comparison. In Section \ref{sec:globoMetric}, we provide a summary of some globalization results from the metric setting that we wish to mirror with our `Lorentzified' constructions. In particular, we give explicit statements of the so-called Alexandrov's Patchwork, Toponogov's Theorem, and the Bonnet--Myers Theorem. 
The main results of this paper are proven within Section \ref{sec:globoLorentz}, which is split into two parts. The first part concerns globalization of upper timelike curvature bounds via gluing of timelike triangles and culminates in the following theorem: 

\begin{introtheorem}{\ref{thm: Lor AlexPatch}}[Alexandrov's Patchwork Globalization, Lorentzian version]
Let $X$ be a strongly causal, non-timelike locally isolating, and regular \LpLS which has (local) timelike curvature bounded above by $K\in\mb{R}$. Suppose that $X$ satisfies \emph{(i)} and \emph{(ii)} in Definition \ref{TLCB}. Additionally assume that the geodesics between timelike related points with $\tau$-distance less than $D_K$ are unique. Let $G$ be the geodesic map of $X$ restricted to the set
$\Set*{ (x,y,t)\in{\ll}\times[0,1] \given \tau(x,y) < D_K}=\tau^{-1}((0,D_K))\times[0,1] $ and assume that $G$ is continuous. 
Then $X$ also satisfies Definition \ref{TLCB}.\emph{(iii)}, in particular $X$ has global curvature bounded above by $K$. 
\end{introtheorem}

As in the metric setting, several of the assumptions in the above theorem are not only necessary, but also sufficient. In particular, we show that every strongly causal and regular Lorentzian pre-length space with global timelike curvature bounded above by $K\in \mb{R}$ has unique geodesics up to length $D_K$.

\begin{introtheorem}{\ref{thm: Lor unique geodesics}}[Unique geodesics in upper curvature bounds]
Let $X$ be a strongly causal and regular Lorentzian pre-length space with timelike curvature bounded above by $K \in \mathbb{R}$. Let $x\ll y$ be in a comparison neighbourhood $U \subseteq X$ and suppose $\tau(x,y)<D_K$. Then there exists a unique geodesic from $x$ to $y$ contained in $U$. In particular, if $X$ satisfies a global upper curvature bound, geodesics between timelike related points in $X$ with $\tau$-distance less than $D_K$ are unique. 
\end{introtheorem}

The second part concerns a result akin to the Bonnet--Myers Theorem, bounding the finite timelike diameter of \LpLSsn, using global lower timelike curvature bounds. More precisely:

\begin{introtheorem}{\ref{thm: lor meyers}}[Bound on the finite diameter]
Let $X$ be a strongly causal, locally causally closed, regular, and geodesic Lorentzian pre-length space which has global curvature bounded below by $K<0$. Assume that, for each pair of points $x\ll z$ in $X$, there exists $y \in X$ such that $\Delta(x,y,z)$ is a non-degenerate timelike triangle. Then $\diam_{\mathrm{fin}}(X)\leq D_K$.
\end{introtheorem}

We conclude the paper with a discussion of ongoing research into the globalization of timelike curvature bounded below, as well as highlighting a number of potential direct applications of Lorentzian pre-length spaces and global curvature bounds, including to the theory of causal sets.

\section{Preliminaries}\label{sec:prelim}

In this section we collect basic results from the theory of Lorentzian length spaces that will be of use in this article. For more details, we refer the interested reader to \citet*{KS18}. We also recall the corresponding elementary concepts from metric geometry, as a showcase of the tools used in the globalization of metric curvature bounds. For details regarding their precise application, see \citet*{BH99, BBI01}. 

\subsection{Introduction to Lorentzian pre-length spaces}

Let us begin by summarising the fundamentals of the Lorentzian length space framework, pioneered by \citet*{KS18}. In particular, we present rungs from the causal ladder which will be necessary in later proofs, in addition to describing the use of triangle comparison to test for curvature bounds. See \citet*{ACS20,Rot22} for discussions on the causal ladder as a whole. First, let us define a Lorentzian pre-length space:

\begin{defi}[Lorentzian pre-length space]
Let $(X,d)$ be a metric space, $\ll, \leq$ two relations on $X$, and $\tau:X \times X \to [0,\infty]$ a function. The quintuple $\lpls$ is then called a \emph{\LpLS} if it satisfies the following:
\begin{itemize}
\item[(i)] $(X,\ll,\leq)$ is a \emph{causal space}, i.e., $\leq$ is a reflexive and transitive relation and $\ll$ is a transitive relation contained in $\leq$.
\item[(ii)] $\tau$ is lower semi-continuous with respect to $d$.
\item[(iii)] $\tau(x,z) \geq \tau(x,y) + \tau(y,z)$ for $x \leq y \leq z$ and $\tau(x,y) >0 \iff x \ll y$.
\end{itemize}
In this case, $\tau$ is called the \emph{time separation function}, with $\ll$ and $\leq$ referred to as the \emph{timelike} and \emph{causal relations}, respectively. All of these concepts are motivated by the corresponding notions in spacetimes.
\end{defi}

A \LpLS $\lpls$ will usually be denoted simply by $X$, where the latter is clear. When referring to causal (or timelike) pasts and futures, we shall use the standard notation, e.g., $I^+(x)\coloneqq \Set*{ y \in X \given x \ll y }$ and $J^-(x)\coloneqq \Set*{ y\in X \given y\leq x }$. In particular, for diamonds, we shall use the notation $J(x,z):=J^+(x) \cap J^-(z)= \Set*{ y \in X \given x\leq y \leq z }$.

\begin{defi}[Causal and timelike curves]
Let $X$ be a \LpLSn. A locally Lipschitz curve $\gamma:[a,b] \to X$ is called \emph{future-directed causal} (respectively \emph{timelike}), if $\gamma(s) \leq \gamma(t)$ (respectively $\gamma(s) \ll \gamma(t)$) for all $s<t$ in $[a,b]$. A \emph{past-directed} curve is defined analogously with the relations in $X$ reversed. To make our terminology less cumbersome and avoid repeated reference to time orientation, we assume causal curves are future-directed, unless it is explicitly stated otherwise.
\end{defi}

\begin{defi}[$\tau$-length and geodesics]
Let $\gamma:[a,b] \to X$ be a causal curve from $x$ to $y$ in a \LpLS $X$. 
\begin{itemize}
\item[(i)] We define its \emph{$\tau$-length} as
\begin{equation}
L_{\tau}(\gamma):=\inf \Set*{ \sum_{i=0}^{n-1} \tau(\gamma(t_i),\gamma(t_{i+1})) \given a= t_0 < t_1 < \ldots < t_n = b, n \in \mathbb{N} }.
\end{equation}

\item[(ii)] By definition we always have $L_{\tau}(\gamma) \leq \tau(x,y)$. In the case of equality, we call $\gamma$ a \emph{distance-realizer} or \emph{geodesic}. 
 
\item[(iii)] $X$ is called \emph{geodesic} if, for each pair of causally related points, there exists a geodesic connecting them. $X$ is called \emph{uniquely geodesic} if it is geodesic and the geodesic between each pair of points is unique.
If a geodesic $\gamma$ between $x \ll y$ is unique, or if it is not unique and a choice of geodesic has either been made or is inconsequential, then we denote the geodesic by $\gamma_{xy}$. Unless otherwise mentioned, we shall assume $\gamma_{xy}$ is parameterized by $[0,1]$ and has constant speed, i.e., $\tau(\gamma(s), \gamma(t))= \tau(x,y)|t-s|$ for all $s<t \in [0,1]$. 
\end{itemize}
\end{defi} 

The additional assumptions given above, that geodesics are parameterized by $[0,1]$ and have constant speed, pose no technical issues when dealing with geodesics (between timelike related points) which are timelike. However, in general, a geodesic between timelike related points may contain a null piece, see for example the causal funnel in \citet*[Example 3.19]{KS18}. In formulating the main results of this paper, we will consider Lorentzian pre-length spaces which are well-behaved, in the sense that geodesics between timelike related points are always timelike. In order to encode this as a property of the space, we now introduce the notion of regularity.

\begin{defi}[Regular Lorentzian pre-length space]
\label{def: regular}
A \LpLS $X$ is called \emph{regular} if for all $x,y \in X$ such that $x \ll y$ all geodesics connecting $x$ and $y$ are timelike.
\end{defi}

It is worth observing that under strong causality, the notions of being regularly localizable, cf.\ \citet*[Definition 3.16]{KS18}, is equivalent to being regular (in the sense of Definition \ref{def: regular}) and localizable, see  the recent work of \citet*[Lemma 3.6]{BKR23}.  

Before providing a definition of timelike curvature bounds, it is now necessary to introduce the notion of triangle comparison:

\begin{defi}[Model spaces and triangle comparison]
\label{def: tr}
Let $X$ be a \LpLSn. We define the following:
\begin{itemize}
\item[(i)] A \emph{timelike triangle} $\Delta(x,y,z)$ in $X$ is a collection of three timelike related points $x \ll y \ll z$ and three pairwise connecting geodesics $\gamma_{xy},\gamma_{yz}$ and $\gamma_{xz}$ between them. To indicate a point $p$ lies on the triangle, we write $p \in \Delta(x,y,z)$, with $p \in \gamma_{xy}$ used to specify which side.

\item[(ii)] By $\lm{K}$ we denote the Lorentzian model space of constant sectional curvature $K$. 
That is, $\lm{0}$ is the Minkowski plane and $\lm{K}$, for $K>0$ and $K<0$, is an appropriately scaled version of 2-dimensional de Sitter or anti-de Sitter spacetime, respectively.

\item[(iii)] Let $\Delta(x,y,z)$ be a timelike triangle in $X$. We call a timelike triangle $\Delta(\bx,\by,\bz)$ in $\lm{K}$ whose sides have the same sidelengths as $\Delta(x,y,z)$ a \emph{comparison triangle} for $\Delta(x,y,z)$. 

\item[(iv)] Let $p\in\gamma_{xy}$ (analogously for $\gamma_{xz}$ and $\gamma_{yz}$) be a point on some side of the triangle $\Delta(x,y,z)$ in $X$ and let $\Delta(\bx,\by,\bz)$ be a comparison triangle for $\Delta(x,y,z)$. The \emph{comparison point} for $p$, in $\Delta(\bx,\by,\bz)$, is the (unique) point $\bp\in\gamma_{\bx\by}$, whose $\tau$-distance to the endpoints of $\gamma_{\bx\by}$ is the same as the $\tau$-distance from $p$ to the respective endpoints of $\gamma_{xy}$. 

\item[(v)] We call a timelike triangle $\Delta(x,y,z)$ \emph{non-degenerate} if the inequality $\tau(x,z) > \tau(x,y) + \tau(y,z)$ holds. In this case, any associated comparison triangle is non-degenerate in the visual sense, i.e. not just a geodesic segment.
\end{itemize}
\end{defi}

Before finally providing the definition of timelike curvature bounds, we raise the following technical detail. 
We feel that in general, but especially in the context of this paper, the original definition of timelike curvature bounds given by \citet*{KS18} should be 
modified slightly. More precisely, the defining properties of a comparison neighbourhood (see \citet*[Definition 4.7]{KS18}) should only hold wherever $\tau$ is ``not too large'' for the comparison space. This is mainly as a result of exotic behaviour exhibited by anti de-Sitter space (AdS), the model space for constant negative timelike curvature. In AdS, geodesics (in the smooth sense) stop being maximizing when they exceed length $\pi$ (when $K=-1$, otherwise this bound is appropriately scaled). 
\bigskip

Visually, this can be explained as follows: place two points $x$ and $y$ on the ``equator'' of AdS such that they are not antipodal and connect them via the longer geodesic (recall that geodesics in AdS arise by intersecting the space with a plane through 0 and the two endpoints). 
Then we can create curves of increasing lengths by going out to infinity, say to the right of the equator. 
In particular, there is no longest curve joining $x$ and $y$ and $\tau(x,y)=\infty$. 
This is related to AdS not being globally hyperbolic; indeed, the maximal globally hyperbolic subset of AdS has (ordinary) diameter $\pi$.
Due to the dependence on global hyperbolicity, this pathological behaviour is clearly exclusive to the Lorentzian case, however there is some similarity with the Maximal Diameter Theorem in Riemannian geometry, see \citet*[Theorem 6.5]{CE75}. An analogue to this result in the setting of Lorentzian manifolds was provided by \citet*[Theorem 3.5]{Har82}.
More precisely, in the metric model spaces $M_k$ with $k>0$, there exist infinitely many geodesics between points which are exactly a distance $\pi / \sqrt{k}$ apart (compare the antipodal points on AdS with those on the sphere). 
\bigskip

In order to provide our updated definition, we have to introduce the so-called finite diameter of a space. This is essentially the diameter, i.e., the supremum of all values of $\tau$, but we explicitly exclude $\infty$ as a value because of the nature of AdS. Note that, despite the nomenclature, the finite diameter of a \LpLS need not be finite.
\begin{defi}[Finite diameter]
\label{def: fin diam}
Let $X$ be a Lorentzian pre-length space. 
\begin{itemize}
    \item[(i)] The \emph{finite diameter} of $X$ is
    \begin{equation}
    \diamfin(X)=\sup(\Set*{\tau(x,y) \given x\ll y}\setminus\{\infty\}),    
    \end{equation}
    i.e., the supremum of all values $\tau$ takes except $\infty$. 
    \item[(ii)] By \emph{$D_K$} we denote the finite diameter of $\lm{K}$. In particular, 
    \begin{equation}
    D_K = \diamfin(\lm{K}) = 
        \begin{cases}
        \infty, \text{ if $K \geq 0$}, \\
        \frac{\pi}{\sqrt{-K}}, \text{ if $K<0$}.
        \end{cases}
    \end{equation}
\end{itemize}
\end{defi}

\begin{rem}[Size Bounds]
Let $K \in \mathbb{R}$. A triple $(a,b,c) \in \mathbb{R}^3_+$ with $c \geq a + b$ are said to satisfy \emph{size-bounds} for $\lm{K}$ if they may be realized as the side lengths of a timelike triangle in $\lm{K}$. 
In particular, the side lengths of a triangle $\Delta(x,y,z)$ in $X$ satisfy size-bounds precisely if $\tau(x,z)<D_K$, cf.\ \citet*[Lemma 4.6]{KS18}, in which case a comparison triangle $\Delta(\bx,\by,\bz)$ exists in $\lm{K}$. For the remainder of this paper, we assume timelike triangles satisfy the required size-bounds for the existence of the comparison triangles used.
\end{rem}

Let us now make our point concrete: \emph{all} properties of a comparison neighbourhood should respect the appropriate range of values of $\tau$. In particular, we do not care whether $\tau$ is continuous near points separated by a distance which cannot be realized in the model space, and we should also not require such points to possess a joining geodesic. 
On the one hand, this refined definition is somehow more in alignment with its metric counterpart. For example, in the definition of CAT($k$) spaces, cf.\ \citet*[Definition II.1.1]{BH99}, the authors explicitly only require that there exists a geodesic between points which are less than the diameter of the corresponding model space apart. 
On the other hand, curvature bounds should morally not be concerned with behaviour which cannot be realized in the model space. 

\begin{defi}[Timelike curvature bounds]
\label{TLCB}
Let $X$ be a \LpLSn. An open subset $U$ is called a \emph{timelike $(\geq K)$-comparison neighbourhood} (or \emph{timelike $(\leq K)$-comparison neighbourhood}) if:
\begin{enumerate}
\item $\tau$ is continuous on $(U\times U) \cap \tau^{-1}([0,D_K))$ and $(U\times U) \cap \tau^{-1}([0,D_K))$ is open. 

\item For all $x,y \in U$ with $x \ll y$ and $\tau(x,y) < D_K$ there exists a geodesic connecting them which is contained entirely in $U$.
\item \label{TLCB.item3} Let $\Delta (x,y,z)$ be a timelike triangle in $U$,
with $p,q$ two points on the sides of $\Delta (x,y,z)$. Let 
$\bar\Delta(\bx, \by, \bz)$ be a comparison triangle in $\lm{K}$ for $\Delta (x,y,z)$ and $\bp,\bq$ comparison points for $p$ and $q$, respectively. Then
\begin{equation}
\label{eq: triangle comparison}
\tau(p,q) \leq \tau(\bp,\bq) \quad \text{ (or } \tau(p,q) \geq \tau(\bp, \bq) \text{)}.
\end{equation}
\end{enumerate}

We say $X$ has \emph{timelike curvature bounded below by $K$} if it is covered by timelike $(\geq K)$-comparison neighbourhoods. Likewise, $X$ has \emph{timelike curvature bounded above by $K$} if it is covered by timelike $(\leq K)$-comparison neighbourhoods.

We say $X$ has \emph{global timelike curvature bounded below by $K$} if $X$ itself is a $(\geq K)$-comparison neighbourhood. Similarly, $X$ has \emph{global timelike curvature bounded above by $K$} if $X$ is a $(\leq K)$-comparison neighbourhood.
\end{defi}
 
 Note that within a $(\geq K)$ comparison neighbourhood, $p \ll q$ implies $\bp \ll \bq$, and within a $(\leq K)$ comparison neighbourhood, $\bp \ll \bq$ implies $p \ll q$. 

\begin{rem}[Global curvature bound of AdS]
     Note that with the above definition, $\lm{-1}$ satisfies global curvature bounds both above and below. In contrast, $\lm{-1}$ does not satisfy a global curvature bound with respect to the original definition of \citet*{KS18}, since it does not satisfy their conditions for a comparison neighbourhood: $\tau$ is neither finite nor continuous, and for $x,y$ with $\tau(x,y) > \pi$, there is no geodesic joining them. 
\end{rem}

When treating local curvature bounds, we consider spaces covered by comparison neighbourhoods. In establishing a globalization theorem, a precise description of the aforementioned covering will be useful. To this end, one step of the causal ladder, namely strong causality, is crucial:

\begin{defi}[Strong causality]
A \LpLS $X$ is called \emph{strongly causal} if $\mathcal{I}:= \Set*{ I(x,y) \given x,y \in X}$ is a subbasis for the topology induced by $d$.
\end{defi}

It turns out, however, that a finite intersection of diamonds inside an arbitrary neighbourhood is not sufficient for subdividing arbitrary timelike triangles into sub-triangles as needed in the eventual Lorentzian version of Alexandrov's Patchwork; we will actually require the existence of a single timelike diamond inside any neighbourhood, i.e., $\mathcal{I}$ must be a basis for the topology. In order to show that such a basis exists, we now introduce the notions of non-timelike locally isolating spaces and approximating spaces in the sense of \citet*[Definition 3.2.3]{BR22} and \citet*[Definition 2.17]{BGH21} respectively:

\begin{defi}[Non-timelike local isolating spaces] 
A subset $A$ of a Lorentzian pre-length space $X$ is said to be \emph{non-future locally isolating} if for all $a \in A$ with $I^+(a) \neq \emptyset$ and for all neighbourhoods $U_a \subseteq A$ of $a$ there exists $b_+ \in U_a$ such that $a \ll b_+$. Similarly, we define a non-past locally isolating set. We say $A$ is non-timelike locally isolating if it satisfies both properties.
\end{defi}

\begin{defi}[Approximating spaces]
A \LpLS $X$ is said to be \emph{(future-/past-) approximating} if for every point $p \in X$ there exists a sequence $(p^{\pm}_n)_n$ in $I^{\pm}(p)$ such that $p^{\pm}_n \to p$ as $n \to \infty$. 
\end{defi}

Note that if every point in $X$ has non-empty timelike future and past, e.g., if $X$ is strongly causal, $X$ (viewed as a subset of itself) is non-timelike local isolating if and only if $X$ is approximating. For the remainder of the paper, we shall work in such a setting, hence we shall use the former terminology when referring to either definition.

\begin{pop}[Diamonds form basis]
\label{pop: tl dia basis}
Let $X$ be a strongly causal and non-timelike locally isolating \LpLSn. Then $\mathcal{I}$ forms a basis for the topology. In particular, given any neighbourhood of any point, we can construct a timelike diamond containing the point, such that the diamond and its governing points are also contained in the neighbourhood.
\end{pop}

\begin{proof}
See \citet*[Lemma 3.5]{Rot22}.
\end{proof}

Conveniently, the previous proposition also proves to be the perfect tool for highlighting the relationship between comparison neighbourhoods in the sense of \citet*[Definition 4.7]{KS18} and the modification we propose in Definition \ref{TLCB}. Indeed, under assumptions of chronology,\newfoot{Recall that a \LpLS $X$ is called \emph{chronological} if $\ll$ is irreflexive. This is the same as requiring $\tau(x,x)=0$ for all $x\in X$.} strong causality, and non-timelike local isolation (as in AdS, for example), the two definitions may be used interchangeably and one may construct comparison neighbourhoods in the sense of \citet*{KS18} via the following lemma: 

\begin{lem}[Automatic size bounds]
\label{lem: autoSizeBounds}
Let $X$ be a chronological, strongly causal, and non-timelike locally isolating \LpLS which has timelike curvature bounded below (above) by some $K\in\mb{R}$. Then $X$ is covered by timelike $(\geq K)$-comparison (resp.\ $(\leq K)$-comparison) neighbourhoods $U$ where $\tau|_{U\times U}<D_K$. In particular, these $U$ are curvature comparison neighbourhoods in the sense of \citet*[Definition 4.7]{KS18} and all timelike triangles contained in such a $U$ are realizable.

Moreover, all open $(\geq K)$-comparison (resp. $(\leq K)$-comparison) neighbourhoods in the sense of \citet*[Definition 4.7]{KS18} are $(\geq K)$-comparison (resp. $(\leq K)$-comparison) neighbourhoods in the sense of Definition \ref{TLCB}. In particular, a chronological, strongly causal, and non-timelike locally isolating \LpLS has (local) curvature bounds in the sense of Definition \ref{TLCB} if and only if it has curvature bounds in the sense of \citet*[Definition 4.7]{KS18}.
\end{lem}

\begin{proof}
Let $x\in X$ and $\tilde{U}$ be a $(\geq K)$-comparison (resp.\ $(\leq K)$-comparison) neighbourhood of $x$. We have that $\tau(x,x)=0$, since $X$ is chronological, hence $(\tilde{U}\times\tilde{U})\cap\tau^{-1}([0,D_K))$ is open and contains $(x,x)$. Consequently, we find a small neighbourhood $V$ of $x$ such that $V\times V\subseteq (\tilde{U}\times\tilde{U})\cap\tau^{-1}([0,D_K))$. By Proposition \ref{pop: tl dia basis}, we find $x_-,x_+ \in V$ such that  $x_-\ll x\ll x_+$ and $x\in I(x_-,x_+)\subset V$ and set $U=I(x_-,x_+)$. It follows that $\tau(x_-,x_+)<D_k$, hence $\tau|_{U\times U}<D_k$. 

We now verify that $U$ is a $(\geq K)$-comparison (resp.\ $(\leq K)$-comparison) neighbourhood in the sense of \citet*[Definition 4.7]{KS18}: Clearly $\tau$ is finite and continuous on $U\times U$. Furthermore, by causal convexity\newfoot{Recall that a subset $A$ of a Lorentzian pre-length space $X$ is called \emph{causally convex} if for all $p,q \in A$ it holds that $J(p,q) \subseteq A$. Causal and timelike diamonds are among the most prominent examples of causally convex sets.} of timelike diamonds, geodesics between points in $U$ remain in $U$. Finally, $U$ inherits property \emph{(iii)} of Definition \ref{TLCB} from the comparison neighbourhood $\tilde U$. 

Next, we show that all open $(\geq K)$-comparison (resp.\ $(\leq K)$-comparison) neighbourhoods in the sense of \citet*[Definition 4.7]{KS18} are $(\geq K)$-comparison (resp.\ $(\leq K)$-comparison) neighbourhoods in the sense of Definition \ref{TLCB}. Properties \emph{(ii)} and \emph{(iii)} are automatic, so it remains to check property \emph{(i)}. As $\tau$ is continuous on $U\times U$, it is also continuous on the subset $(U\times U)\cap \tau^{-1}([0,D_K))$. Furthermore, $[0,D_K)$ is open in $[0,\infty]$, so it follows from continuity of $\tau$ on $U\times U$ that $(U\times U)\cap \tau^{-1}([0,D_K))$ is open in $U\times U$ and therefore in $X\times X$. 
\end{proof} 

In metric geometry, there are several reformulations of curvature bounds expressed via classical triangle comparison, using angles, for example. Alternative formulations of timelike curvature bounds have also been derived (see \citet*{BS22,BMS22}) and several of these characterizations will prove useful in our context. Before we state these explicitly, let us introduce some more terminology:

\begin{defi}[$K$-comparison angles and sign]
Let $X$ be a \LpLSn, $K \in \mb{R}$, $\Delta(x,y,z)$ a timelike triangle in $X$, and $\Delta(\bx,\by,\bz)$ a comparison triangle in $\lm{K}$ for $\Delta(x,y,z)$.

\begin{itemize}
    \item[(i)] The \emph{$K$-comparison angle} at $x$ is defined as the ordinary hyperbolic angle at $\bx$ between $\by$ and $\bz$:
\begin{equation}
    \tilde{\ma}_x^{K}(y,z):=\ma_{\bx}^{\lm{K}}(\by,\bz)=\arcosh(|\langle \gamma_{\bx \by}'(0), \gamma_{\bx \bz}'(0) \rangle|), 
\end{equation}
where we assume the mentioned geodesics to be unit speed parameterized.
\item[(ii)] The \emph{sign} $\sigma$ of a $K$-comparison angle is the sign of the corresponding inner product (in the $-,+,\cdots,+$ convention). That is, in this notation, the sign is $-1$ if the angle is measured at $x$ or $z$ and $1$ if the angle is measured at $y$.
\item[(iii)] The \emph{signed $K$-comparison angle} is defined as $\tilde{\ma}_x^{K,S}(y,z):=\sigma \tilde{\ma}_x^{K}(y,z)$.
\end{itemize}
\end{defi}

\begin{defi}[Angles and hinges]
Let $X$ be a \LpLS and let $\alpha$ and $\beta$ be two timelike curves of arbitrary time orientation emanating at $\alpha(0)=\beta(0) \eqqcolon x$. 
\begin{itemize}
    \item[(i)] The \emph{angle} between $\alpha$ and $\beta$, if it exists, is defined as 
    \begin{equation}
    \label{eq: angle}
        \ma_x(\alpha,\beta) \coloneqq \lim_{s,t \to 0}\tilde{\ma}_x^{0}(\alpha(s),\beta(t)), 
    \end{equation}
    where the limit only considers values of $s$ and $t$ for which the triple $(x, \alpha(s),\beta(t)) $ (or some permutation thereof) forms a timelike triangle.
    \item[(ii)] The \emph{sign} $\sigma$ of an angle is $-1$ if $\alpha$ and $\beta$ have the same time orientation and $1$ otherwise. The \emph{signed angle} is defined as $\ma_x^S(\alpha,\beta) \coloneqq \sigma \ma_x(\alpha,\beta)$.
\item[(iii)] An angle at a point $x\in X$ and the associated geodesics are called a \emph{hinge}, which we will denote by $(\alpha, \beta)$.
    \item[(iv)] Given $K\in \mb{R}$ and a hinge $(\alpha,\beta)$ at $x$ in $X$, we call a hinge $(\bar{\alpha},\bar{\beta})$ at $\bx$ in $\lm{K}$ whose corresponding sides have the same lengths and time orientations and satisfy $\ma_x(\alpha,\beta)=\ma_{\bx}^{\lm{K}}(\bar{\alpha}, \bar{\beta})$ a $K$-comparison hinge for $(\alpha,\beta)$.
\end{itemize}
\end{defi}

As we will never work with different model spaces simultaneously and as the limit in \eqref{eq: angle}, cf.\ \citet*[Proposition 2.14]{BS22}, is the same regardless of the model space in which it is considered, we will drop the superscript in the comparison angle and write $\tilde{\ma}_x(y,z)$. 
\bigskip

Now we highlight some alternative formulations of curvature bounds, beginning with monotonicity comparison. The remaining results in this subsection are also valid when considering upper curvature bounds, where inequalities are switched in the obvious way. However, as we will only need our reformulations in the setting of lower curvature bounds, we shall not state the former case herein.
\bigskip

As stated in \citet*[Definition 4.8]{BS22}, monotonicity comparison (specifically point \emph{(ii)}) requires the additional technical assumption that the neighbourhoods considered are \emph{strictly timelike geodesic}, meaning that for any two close enough timelike related points there is a \emph{timelike} geodesic joining them. This can be omitted by instead assuming that our \LpLS is regular as in Definition \ref{def: regular}; monotonicity comparison then takes the following form:

\begin{defi}[$K$-Monotonicity comparison]
Let $K \in \mb{R}$ and let $X$ be a regular \LpLSn. $X$ is said to satisfy \emph{$K$-monotonicity comparison from below} if every point in $X$ possesses an open neighbourhood $U$ such that:

\begin{itemize}
    \item[(i)] $\tau$ is continuous on $(U\times U) \cap \tau^{-1}([0,D_K))$ and $(U\times U) \cap \tau^{-1}([0,D_K))$ is open. 
    
    \item[(ii)] For all $x,y \in U$ with $x \ll y$ and $\tau(x,y) < D_K$ there is a (timelike) geodesic joining them which is contained entirely in $U$. 
    
    \item[(iii)] Given two timelike geodesics $\alpha, \beta :[0,1] \to X$ of arbitrary time orientation emanating at $\alpha(0)=\beta(0)=x$, we have that $\tilde{\ma}_x^{K,S}(\alpha(s), \alpha(t))$ is a monotonically increasing function in $s$ and $t$ (where it is defined).
\end{itemize}    
\end{defi}

Note that by assuming that our space is regular, points \emph{(i)} and \emph{(ii)} above are precisely those given in Definition \ref{TLCB} of timelike curvature bounds and only point \emph{(iii)} differs.
\bigskip

Currently, monotonicity comparison is the only formulation which is equivalent to triangle comparison in reasonable generality, with other formulations being implied by, but not implying, the monotonicity condition.\newfoot{\citet*[Theorem 14]{BMS22} show that classical triangle comparison can be deduced from angle comparison in the case of lower timelike curvature bounds using additional assumptions on the behaviour of angles.}

\begin{thm}[Triangle and monotonicity comparison are equivalent]
\label{thm: monotonicity}
Let $K \in \mb{R}$ and let $X$ be a regular \LpLSn. Then $X$ has timelike curvature bounded below by $K$ in the sense of Definition \ref{TLCB} if and only if it satisfies $K$-monotonicity comparison from below.
\end{thm}
\begin{proof}
    See \citet*[Theorem 4.13]{BS22}.
\end{proof}

\begin{thm}[Curvature bounds imply angle and hinge comparison]
\label{thm: equivalent curv bounds}
    Let $X$ be a regular \LpLS with timelike curvature bounded below by $K \in \mb{R}$. Let $x \in X$ and let $\alpha, \beta: [0,1] \to X$ be any two timelike geodesics emanating from $x$.

    \begin{itemize}
        \item[(i)] It holds that 
        \begin{equation}
            \ma_x^S(\alpha,\beta) \leq \tilde{\ma}_{x}^S(\alpha(s),\beta(t))
        \end{equation}
        for all $s,t$ which form a timelike triangle with $x$.
        \item[(ii)] Let $(\bar{\alpha}, \bar{\beta})$ be a comparison hinge in $\lm{K}$. Then 
        \begin{equation}
        \tau(\alpha(1),\beta(1)) \geq \bar{\tau}(\bar{\alpha}(1), \bar{\beta}(1)).  \end{equation}
        \end{itemize}
\end{thm}
\begin{proof}
    See \citet*[Corollaries 4.11 and 4.12]{BS22}.
\end{proof}

We conclude this chapter with the following useful fact about angles.
\begin{pop}
\label{pop: equal angles along geodesic}
    Let $X$ be a strongly causal and locally causally closed \LpLS with timelike curvature bounded below by $K \in \mathbb{R}$, and let $\alpha:[0,1] \to X$ be a timelike geodesic. Let $x=\alpha(t)$ for $t \in (0,1)$ and consider the restrictions $\alpha_-:=\alpha|_{[0,t]}(t-\cdot)$ and $\alpha_+ :=\alpha|_{[t,1]}$ as past-directed and future-directed geodesics emanating from $x$, respectively. 
    Let $\beta$ be a timelike geodesic emanating from $x$. Then $\ma_x(\alpha_-,\beta)=\ma_x(\alpha_+,\beta)$. 
\end{pop}
\begin{proof}
    See \citet*[Corollary 4.6, Lemma 4.10]{BS22} for the proof of the equality and for the existence of the angle, respectively.
\end{proof}

\subsection{Elementary concepts from metric geometry}

We now turn to presenting some basic definitions from the realm of metric geometry. Most importantly, we would like to highlight the fundamental differences in the definitions of lengths and curvature bounds in the metric setting when compared to those Lorentzian pre-length spaces. 

\begin{defi}[Length of a curve and geodesics]
Let $(X,d)$ be a metric space. 
The \emph{length} of a curve $\gamma: [a,b] \to X$ from $x$ to $y$ is defined as 
\begin{equation}
L_d(\gamma):=\sup \Set*{ \sum_{i=0}^{n-1} d(\gamma(t_i),\gamma(t_{i+1})) \given a=t_0 < t_1 < \ldots < t_n = b, n \in \mathbb{N}}.
\end{equation}
If $L_d(\gamma)=d(x,y)$, then $\gamma$ is called a \emph{distance-realizer} or \emph{geodesic}. 
\end{defi}

We define triangles, comparison triangles, and comparison points in complete analogy to Definition \ref{def: tr}. Again we assume that all triangles satisfy size bounds. We will generally not use Lorentzian and Riemannian model spaces simultaneously, however we shall denote the Riemannian model spaces by $M_k$ for clarity, cf.\ \citet*{CE75}. Similarly, we will not consider metric and Lorentzian geodesics  simultaneously, so context should be sufficient to deduce which of the two concepts is being applied.

\begin{defi}[Metric triangle comparison]
Let $X$ be a metric space. An open subset $U$ is called a \emph{$(\geq K)$-comparison neighbourhood} (or \emph{$(\leq k)$-comparison neighbourhood}) if $(U,d|_{U \times U})$ is geodesic for pairs of points with distance less than $\diam(M_k)$, and for all triangles $\Delta(x,y,z)$ (satisfying size bounds) in $U$, and all $p,q \in \Delta(x,y,z)$, the following is satisfied: let $\bar{\Delta}(\bx,\by,\bz)$ be a comparison triangle in $M_k$ for $\Delta(x,y,z)$  and let $\bp,\bq \in \bar{\Delta}(\bx,\by,\bz)$ be comparison points for $p$ and $q$ respectively. Then
\begin{equation}
\label{eq: metric curv bds}
    d(p,q) \geq d(\bar{p},\bar{q}) \quad \text{ (or } d(p,q) \leq d(\bar{p},\bar{q})\text{)}\,.
\end{equation}
We say $X$ has \emph{curvature bounded below by $k$} if it is covered by $(\geq k)$-comparison neighbourhoods. Likewise, its \emph{curvature is bounded above by $k$} if it is covered by $(\leq k)$-comparison neighbourhoods. 
\bigskip

We say $X$ has global curvature bounded below (or above) by $k$ if $X$ is a $(\geq k)$ (or $(\leq k)$) comparison neighbourhood. 
Spaces with global curvature bounded above by $k$ are called CAT($k$) spaces. 
\end{defi}

\section{Metric spaces with global curvature bounds}\label{sec:globoMetric}
The globalization of curvature bounds in metric spaces serves as a natural motivation for investigating analogous results in the Lorentzian case. Hence, we include here a brief discussion of the metric picture in order to familiarize ourselves and the reader with the techniques and constraints we wish to transfer. For details of the wider metric setting, we refer the reader to the introductory texts by \citet*{BBI01}, \citet*{BH99}, and \citet*{AKP19}.

\subsection{Curvature bounded above}
Let us first consider the case of curvature bounded above. The following example demonstrates that globalization is not automatic in this case, and that some additional assumptions are required. 

\begin{ex}[The circle]\label{ex:metricCircle}
Consider the unit circle $X \coloneqq \mathbb{S}^1$ with its intrinsic metric. Then locally, $X$ is isometric to a line segment, which clearly has curvature bounded above by $k=0$. 

However, $X$ itself is not a $\leq 0$-comparison neighbourhood: consider a large triangle defined by three equidistant points in $X$, such that it covers the whole circle.  
The corresponding comparison triangle in the Euclidean plane is also equilateral. Given any two points $p,q$ on different sides of the triangle in $X$, the triangle inequality yields equality when going along the shorter of the two arcs between each pair of points, see Figure \ref{fig: metric counterexample glob curv bds}. However, in the plane, triangle equality is only obtained when the two points lie on the same side of the triangle. It follows that \eqref{eq: metric curv bds} is not satisfied. 

In summary, $X$ has (local) curvature bounded above by $k=0$, but does not have the corresponding global bound. Therefore, any globalization theorem for curvature bounded above requires assumptions which the circle does not satisfy when $k=0$.

\end{ex} 

\begin{figure}[ht]
\begin{center}
\begin{tikzpicture}[line cap=round,line join=round,>=triangle 45,x=1cm,y=1cm]
\draw (0,0) circle (1cm);
\draw (2.133974596215561,-0.5)-- (3,1);
\draw (3,1)-- (3.866025403784439,-0.5);
\draw[dotted] (2.711324865405187,-0.5)-- (2.4226497308103743,0);
\draw (3.866025403784439,-0.5)-- (2.133974596215561,-0.5);
\draw [shift={(0,0)},line width=1pt]  plot[domain=3.141592653589793:4.1887902047863905,variable=\t]({1*1*cos(\t r)+0*1*sin(\t r)},{0*1*cos(\t r)+1*1*sin(\t r)});
\begin{scriptsize}
\coordinate [circle, fill=black, inner sep=0.5pt, label=90: {$z$}] (by) at (0,1);
\coordinate [circle, fill=black, inner sep=0.5pt, label=90: {$\bz$}] (by) at (3,1);
\coordinate [circle, fill=black, inner sep=0.5pt, label=0: {$y$}] (by) at  (0.8660254037844387,-0.5);
\coordinate [circle, fill=black, inner sep=0.5pt, label=180: {$x$}] (by) at (-0.8660254037844387,-0.5);
\coordinate [circle, fill=black, inner sep=0.5pt, label=180: {$\bx$}] (by) at (2.133974596215561,-0.5);
\coordinate [circle, fill=black, inner sep=0.5pt, label=0: {$\by$}] (by) at (3.866025403784439,-0.5);
\coordinate [circle, fill=black, inner sep=0.5pt, label=270: {$p$}] (by) at (-0.5,-0.8660254037844385);
\coordinate [circle, fill=black, inner sep=0.5pt, label=180: {$q$}] (by) at (-1,0);
\coordinate [circle, fill=black, inner sep=0.5pt, label=270: {$\bp$}] (by) at (2.711324865405187,-0.5);
\coordinate [circle, fill=black, inner sep=0.5pt, label=180: {$\bq$}] (by) at (2.4226497308103743,0);
\end{scriptsize}
\end{tikzpicture}
\caption{Triangle comparison fails for too large triangles in the circle.}
\end{center}
\label{fig: metric counterexample glob curv bds}
\end{figure}
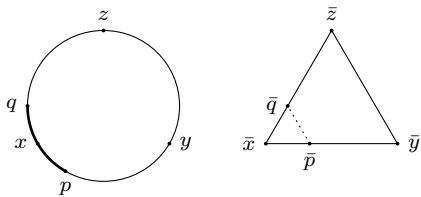

The following theorem specifies sufficient additional assumptions under which upper curvature bounds may be globalized: 
\begin{thm}[Alexandrov's Patchwork]
\label{thm: alex patch}
Let $X$ be a metric space with (local) curvature bounded above by $k$ and suppose that there is a unique geodesic joining each pair of points that are a distance less than $\diam(M_k)$ apart. If these geodesics vary continuously with their endpoints,\newfoot{This means that if $x_n \to x, y_n \to y$, then $\gamma_{x_n y_n} \to \gamma_{xy}$ uniformly.} then $X$ is a CAT($k$) space.
\end{thm}
\begin{proof}
See \citet*[Proposition II.4.9]{BH99}.
\end{proof}

On a related note, these additional assumptions under which globalization of curvature bounds can be applied are rather mild, in the sense that they always hold in a CAT($k$) space. 
In other words, a metric space with curvature bounded above by $k$ is CAT($k$) if and only if it is uniquely geodesic (for points with distance less than $\diam(M_k)$) and these geodesics vary continuously with their endpoints. 

\begin{pop}[Elementary properties of CAT($k$) spaces]
\label{pop: Elementary Catk}
Let $X$ be a CAT($k$) space. Then geodesics in $X$ are unique between points with distance less than $\diam(M_k)$ and these geodesics vary continuously with their endpoints. 
\end{pop}
\begin{proof}
See \citet*[Proposition II.1.4]{BH99}.
\end{proof}

It should be apparent that the circle with its intrinsic metric does not have continuously varying geodesics, let alone unique geodesics (recall that $\diam(M_0)=\infty$). Therefore, by the above proposition, it cannot be CAT($0$).

\subsection{Curvature bounded below}
Concerning curvature bounded below, there is also a globalization result, which is perhaps even more iconic than the Alexandrov's Patchwork approach for curvature bounded above. It is known as the theorem of Toponogov and was first proven for general complete length spaces by \citet*{BGP92}.
\bigskip

Note that for Theorem \ref{thm: Toponogov} and Theorem \ref{thm: meyers} in the form stated, we follow \citet*{BBI01} and explicitly exclude 1-dimensional spaces. More precisely, if $k>0$, then $X$ must not be isometric to $\mb{R},\, (0, \infty),\, [0,B]$ for any $B > \frac{\pi}{\sqrt{k}}$, or any circle with radius greater than $\frac{1}{\sqrt{k}}$. In \citet*[Proposition 8.44]{AKP19} it is shown that these are precisely the complete length spaces to which Theorem \ref{thm: meyers} does not apply.

\begin{thm}[Toponogov's Globalization Theorem]
\label{thm: Toponogov}
Let $X$ be a complete length space with curvature bounded below by $k$, which is not one of the aforementioned $1$-dimensional spaces. Then $X$ has global curvature bounded below by $k$.
\end{thm}
\begin{proof}
See \citet*[Theorem~10.3.1]{BBI01} for a proof under a local compactness assumption. See \citet*{LS13, AKP19} for more general proofs, as well as a timeline of refinements.
\end{proof}

In the metric case, there is an addendum to Toponogov's Theorem, generalizing the Bonnet--Myers Theorem from Riemannian geometry to the setting of Alexandrov geometry. In essence, the theorem bounds the diameter of a complete length space with local, positive, lower curvature bound in such a way that comparison triangles exist, thus eliminating concern about the existence of triangles in $X$ which are too large.

\begin{thm}[Lower curvature  bounds imply finite diameter]
\label{thm: meyers}
Let $X$ be a complete length space with (local) curvature bounded below by some $k>0$, which is not one of the aforementioned $1$-dimensional spaces. Then $\diam(X) \leq \frac{\pi}{\sqrt{k}}$. 
\end{thm}
\begin{proof}
See \citet*[Theorem~10.4.1]{BBI01}.
\end{proof}

When deriving an analogue of the Bonnet--Myers Theorem for Lorentzian pre-length spaces in Theorem \ref{thm: lor meyers}, we will not precisely follow the approach taken by \citet*[Theorem 10.4.1]{BBI01} in the metric setting. This is because their method views the bound on the diameter of the space as a direct consequence of Toponogov's Globalization Theorem \ref{thm: Toponogov}, which at the time of writing had no synthetic Lorentzian analogue. Consequently, while their result applies to metric spaces with local curvature bounded below,
we derive a diameter bound on Lorentzian pre-length spaces under the slightly stricter assumption that they have global (timelike) curvature bounded below.
The corresponding result assuming only local timelike curvature bounds and a Lorentzian analogue of Toponogov's Globalization Theorem are therefore natural candidates for future research (see Section \ref{sec:concOutlook}).
We also wish to make the reader aware of the recent preprint by \citet*{BHNR23} which discusses these problems in more detail.

\section{\LpLSs with global curvature bounds}
\label{sec:globoLorentz}

We now return fully to the setting of synthetic Lorentzian geometry. As previously mentioned, the main task of this work is to provide Lorentzian versions of the Alexandrov's Patchwork approach to globalizing upper curvature bounds and the Bonnet--Myers Theorem constraining the diameter of spaces with positive lower curvature bounds on sectional curvature, cf.\ Theorem \ref{thm: alex patch} and Theorem \ref{thm: meyers}, respectively. Let us first discuss Alexandrov's Patchwork.

\subsection{Timelike curvature bounded above}

As it turns out, the proof of the metric globalization result, Theorem \ref{thm: alex patch}, may be very straightforwardly adapted to the Lorentzian setting if we respect some minor technicalities. We first provide a few preparatory results, before diving into the proof proper.
First and foremost, we will need the Gluing Lemma for timelike triangles.
\medskip

\begin{lem}[Gluing Lemma for timelike triangles]
\label{gluinglemma}
Let $K \in \mb{R}$, $X$ be a Lorentzian pre-length space, and $U \subseteq X$ be an open subset satisfying Definition \ref{TLCB} $(i)$ and $(ii)$. Let $T_3:= \Delta(x,y,z)$ be a timelike triangle in $U$ and fix a point $p \in \gamma_{xz}$ with $p \ll y$, such that $T_1 := \Delta(x,p,y)$ and $T_2 := \Delta(p,y,z)$ are also timelike triangles,
see Figure \ref{fig: gluing lemma}. Suppose $T_1$ and $T_2$ have timelike curvature bounded above by $K$, i.e. they satisfy \eqref{eq: triangle comparison}. Then the same holds for $T_3$.

Furthermore, the statement remains valid if $p \in \gamma_{xz}$ is chosen such that $y \ll p$, or if $p$ is on either of the other two sides (in which case the timelike relation to the opposite endpoint is automatic). 

\end{lem}
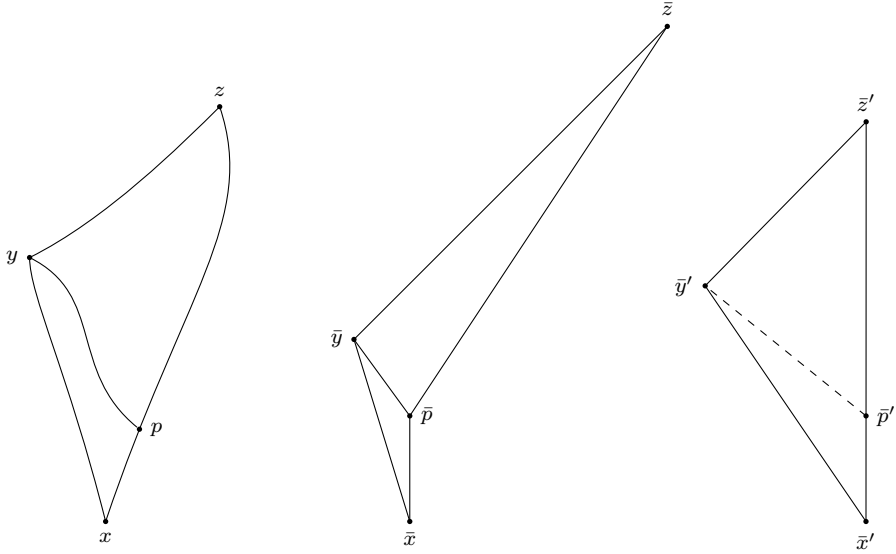
\begin{figure}[ht]
\begin{center}
\begin{tikzpicture}
\draw (0,0)-- (-0.7340132902094361,2.414285714285714);
\draw (-0.7340132902094361,2.414285714285714)-- (0,1.4);
\draw (0,1.4)-- (0,0);
\draw (0,1.4)-- (3.3867606432599557,6.565282921073598);
\draw (3.3867606432599557,6.565282921073598)-- (-0.7340132902094361,2.414285714285714);
\draw (3.8850696994038274,3.124568798471209)-- (6,0);
\draw (3.8850696994038274,3.124568798471209)-- (6.002261946852696,5.300000482679559);
\draw (6.002261946852696,5.300000482679559)-- (6,0);
\draw [dashed] (6.000597495395052,1.4000001275002611) -- (3.8850696994038274,3.124568798471209);

\draw (-4,0) .. controls (-4.5,2) and (-5,3) .. (-5,3.5);
\draw (-5,3.5) .. controls (-4,4) and (-3,5) .. (-2.5,5.5);

\begin{scriptsize}
\draw (-4,0) .. controls (-3,3) and (-2,4) .. (-2.5,5.5) node (A)[circle, fill=black,inner sep=0.7pt,pos=0.15,label=right:$p$]{};;
\draw (A) .. controls (-4.5,2) and (-4,3) .. (-5,3.5);
\coordinate [circle, fill=black, inner sep=0.7pt, label=270: {$x$}] (A1) at (-4,0);
\coordinate [circle, fill=black, inner sep=0.7pt, label=180: {$y$}] (A1) at (-5,3.5);
\coordinate [circle, fill=black, inner sep=0.7pt, label=90: {$z$}] (A1) at (-2.5,5.5);
\coordinate [circle, fill=black, inner sep=0.7pt, label=270: {$\bx$}] (A1) at (0,0);
\coordinate [circle, fill=black, inner sep=0.7pt, label=0: {$\bp$}] (A1) at (0,1.4);
\coordinate [circle, fill=black, inner sep=0.7pt, label=180: {$\by$}] (A1) at (-0.7340132902094361,2.414285714285714);
\coordinate [circle, fill=black, inner sep=0.7pt, label=90: {$\bz$}] (A1) at (3.3867606432599557,6.565282921073598);
\coordinate [circle, fill=black, inner sep=0.7pt, label=270: {$\bx'$}] (A1) at (6,0);
\coordinate [circle, fill=black, inner sep=0.7pt, label=180: {$\by'$}] (A1) at (3.8850696994038274,3.124568798471209);
\coordinate [circle, fill=black, inner sep=0.7pt, label=90: {$\bz'$}] (A1) at (6.002261946852696,5.300000482679559);
\coordinate [circle, fill=black, inner sep=0.7pt, label=0: {$\bp'$}] (A1) at (6.000597495395052,1.4000001275002611);
\end{scriptsize}
\end{tikzpicture}
\end{center}
\caption{On the left, a timelike triangle $\Delta(x,y,z)$ in $X$ is subdivided into two timelike triangles $\Delta(x,p,y)$ and $\Delta(p,y,z)$. In the middle, the comparison triangles $\Delta(\bx,\bp,\by)$ and $\Delta(\bp,\by,\bz)$ for the sub-triangles share the side $\gamma_{py}$. On the right, the comparison triangle $\Delta(\bx',\by',\bz')$ for the outer triangle may be distinct from $\Delta(\bx,\by,\bz)$ in the middle situation.}
\label{fig: gluing lemma}
\end{figure}
\begin{proof}
See \citet*[Lemma 4.3.1, Corollary 4.3.2]{BR22}. Note that this still works with the new definition of timelike curvature bounds.
\end{proof}

So now we know that, if we can triangulate each big timelike triangle into smaller ones, with each of the small sub-triangles contained in a $(\leq K)$-comparison neighbourhood, then we can reconstruct the big timelike triangle step-by-step, using the Gluing Lemma to get that Definition \ref{TLCB}.\ref{TLCB.item3} is satisfied by the big triangle.
\bigskip

In order to globalize curvature bounds, we now require conditions which guarantee such a triangulation of arbitrary triangles, with sub-triangles contained in comparison neighbourhoods. Using Proposition \ref{pop: tl dia basis}, we have the following elegant description of comparison neighbourhoods in strongly causal Lorentzian pre-length spaces with curvature bounds, cf.\ \citet*[Remark 2.2.12]{Ber20}, which shall turn out to be sufficient:

\begin{pop}[Timelike diamonds form neighbourhood basis of comparison neighbourhoods]
\label{pop: tl dia nhood basis comp nhood}
Let $X$ be a strongly causal and non-timelike locally isolating Lorentzian pre-length space with curvature bounded above (or below) by $K \in \mathbb{R}$. Then each point has a neighbourhood basis of timelike diamonds which are also comparison neighbourhoods.
\end{pop}

\begin{proof}
Let $x \in X$ and let $U$ be a comparison neighbourhood of $x$. Any timelike diamond $D:=I(p,q)$ containing $x$ and contained in $U$ (we can even assume $p,q \in U$) is a comparison neighbourhood: Indeed, points \emph{(i)} and \emph{(iii)} of Definition \ref{TLCB} are directly inherited when passing to open subsets, and \emph{(ii)} follows by causal convexity of $D$, as in the proof of Lemma \ref{lem: autoSizeBounds}.

As $X$ is strongly causal and non-timelike local isolating, Proposition \ref{pop: tl dia basis} yields that timelike diamonds form a neighbourhood basis. Hence, for each neighbourhood $V$ of $x$, there exists a timelike diamond containing $x$, which is contained in $U\cap V$. By the above, such a timelike diamond is also a comparison neighbourhood.
\end{proof}

\begin{defi}[Geodesic map]
Let $X$ be a 
uniquely geodesic and regular Lorentzian pre-length space. 
Viewing the timelike relation $\ll$ as a subset of $X \times X$, the \emph{geodesic map}\newfoot{This is closely related to the \emph{line-of-sight map} of \citet*[Definition 9.32]{AKP19}.} of $X$ is formally defined as 
\begin{equation}
G: {\ll} \times [0,1] \to X\,,\quad G(x,y,t):= \gamma_{xy}(t)\,.
\end{equation}
We say that \emph{geodesics vary continuously} if $G$ is continuous. \end{defi}

The definition above appears to be a different notion of continuous variation of geodesics to that required in Theorem \ref{thm: alex patch}. As it turns out, however, they are equivalent (note that of course in the Lorentzian formulation of the version used in Theorem \ref{thm: alex patch} we must restrict to sequences and limits of timelike related points). The following proposition does not make use of $\tau$ at all, but we still formulate it in the Lorentzian context.

\begin{pop}[Equivalent notions of continuously varying geodesics]
Let $X$ be a uniquely geodesic and regular \LpLS and assume $\tau^{-1}([0,D_K))$ is open. 
Then \\ $G|_{\tau^{-1}((0,D_K))\times[0,1]}$ is continuous if and only if (timelike) geodesics with length less than $D_K$ vary continuously in the sense of Theorem \ref{thm: alex patch}. 
\end{pop}

\begin{proof}
First assume we have continuously varying geodesics in the sense of Theorem \ref{thm: alex patch}. 
Note that since geodesics are continuous by definition, we know that $G$ is continuous in $t$. We have to show $G(x_n,y_n,t_n) \to G(x,y,t)$ for sequences $x_n \to x, y_n \to y, t_n \to t, x_n \ll y_n, x \ll y, \tau(x,y)<D_K, \tau(x_n,y_n)<D_K$, i.e., $\gamma_{x_n y_n}(t_n) \to \gamma_{xy}(t)$. We have
$d(\gamma_{x_n y_n}(t_n), \gamma_{xy}(t)) \leq d(\gamma_{x_n y_n}(t_n), \gamma_{xy}(t_n)) + d(\gamma_{xy}(t_n), \gamma_{xy}(t))$ and conclude that the right hand goes to 0 using the uniform convergence $\gamma_{x_n y_n} \to \gamma_{xy}$ and the fact that geodesics are continuous. 

Conversely, suppose (the suitable restriction of) $G$ is continuous. Given sequences $x_n \to x, y_n \to y$ with $\tau(x,y)<D_K$, note that by the openness of $\tau^{-1}([0,D_K))$, we also have $\tau(x_n,y_n)<D_K$ for large enough $n$. 
Note that the set $(\Set*{(x_n,y_n) \given n\in\N }\cup\{(x,y)\})\times[0,1]$ is compact. 
Thus, $G$ restricted to this set is uniformly continuous. 
In particular, the sequence of curves $\gamma_{x_ny_n}=G(x_n,y_n,\cdot)$ is equi-continuous, and as $x_n\to x$ it is also uniformly bounded. Thus, we can apply the Arzelà--Ascoli theorem to get a uniform limit curve $\beta$. By continuity of $G$, we have that $\beta(t)=\lim_n\gamma_{x_ny_n}(t) = \lim_nG(x_n,y_n,t)=G(x,y,t)=\gamma_{xy}(t)$, so the uniform limit is $\gamma_{xy}$.
\end{proof}

\begin{thm}[Alexandrov's Patchwork Globalization, Lorentzian version]
\label{thm: Lor AlexPatch}
Let $X$ be a strongly causal, non-timelike locally isolating, and regular \LpLS which has (local) timelike curvature bounded above by $K\in\mb{R}$. Assume that $X$ satisfies \emph{(i)} and \emph{(ii)} in Definition \ref{TLCB} and that geodesics between timelike related points in $X$ with $\tau$-distance less than $D_K$ are unique.\newfoot{I.e. $\tau^{-1}([0,D_K))$ is open and $\tau$ is continuous on that set, and for all $x \ll y$ in $X$ with $\tau(x,y) < D_K$, there exists unique a geodesic joining them.} Let $G$ be the geodesic map of $X$ restricted to the set
\begin{equation*}
\Set*{ (x,y,t)\in{\ll}\times[0,1] \given \tau(x,y) < D_K}=\tau^{-1}((0,D_K))\times[0,1] 
\end{equation*}
and assume that $G$ is continuous. 
Then $X$ also satisfies Definition \ref{TLCB}.\emph{(iii)}, in particular $X$ has global curvature bounded above by $K$. 
\end{thm}

\begin{proof}
By our assumptions, it is only left to show triangle comparison. 
Let $\Delta(x,y,z)$ be a timelike triangle in $X$. 
Given $t \in [0,1]$, let $\beta_t \coloneqq \gamma_{x \gamma_{yz}(t)}=G(x,\gamma_{yz}(t),\cdot): [0,1] \to X$ be the geodesic from $x$ to $\gamma_{yz}(t)$. By the continuity of $G$, we can regard the map $F(s,t):=\beta_t(s)$ as a geodesic variation with starting point $x$ that ``spans'' the timelike triangle $\Delta(x,y,z)$. 
In particular, this ``filled in'' triangle is compact as the continuous image under $F$ of the compact set $[0,1] \times [0,1]$. 
Fix $t \in [0,1]$. 
For each $s \in [0,1]$, we find a timelike diamond $I(x_s,y_s)$ that is a comparison neighbourhood of $\beta_t(s)$, cf. Proposition \ref{pop: tl dia basis} and Proposition \ref{pop: tl dia nhood basis comp nhood}. 
Since $\beta_t$ is continuous, there is a neighbourhood $N_s$ of $s$ in $[0,1]$ such that $\beta_t(N_s) \subseteq I(x_s,y_s)$.
In particular, for $s \in (0,1)$, we find $s^- < s < s^+$ in $N_s$. By the causal convexity of diamonds, we then obtain that $I_s:=I(\beta_t(s^-),\beta_t(s^+)) \subseteq I(x_s,y_s)$ is also a comparison neighbourhood of $\beta_t(s)$. The point is that we can choose the comparison neighbourhood diamonds in such a way that the governing points are situated on the geodesic. 
For the parameters $0$ and $1$ this will not be possible as these are the endpoints of the geodesic, however, we may still force one of the governing points to be on $\beta_t$, i.e., we set $I_0:=I(x_0,\beta_t(0^+))$ and $I_1:=I(\beta_t(1^-),y_1)$. 
Clearly, $\bigcup_s I_s$ is an open cover of $\beta_t([0,1])$. By compactness, we can extract a finite subcover\newfoot{Due to the manner in which we choose the governing points of each $I_s$, the first and the last diamonds $I_0$ and $I_1$ are always included (because the endpoints are not inside any other $I_s$).} say $\bigcup_{k=0}^n I_{s_k} \supseteq \beta_t([0,1])$.
Now order these diamonds with respect to, say, (the parameters of) their future governing points, i.e., $s_k^+ < s_{k+1}^+$ for all $k$. 
Further assume that the cover is minimal in the sense that no diamond can be removed from the cover; in particular, no diamond is entirely contained inside another one. 
This then immediately implies that the bottom governing points are ordered similarly and that subsequent diamonds overlap and only subsequent ones do so, i.e., 
\begin{equation}
I_{s_i} \cap I_{s_j} \neq \emptyset \iff |i-j| \leq 1.
\end{equation}
Clearly, $F(\cdot,t)=\beta_t$. 
Since $F$ is continuous and $\bigcup_{k=0}^n I_{s_k}$ is a neighbourhood of $\beta_t([0,1])$, it follows that there exists an open neighbourhood of $t$, denote it by $J_t$, such that $F([0,1],J_t) \subseteq \bigcup_{k=0}^n I_{s_k}$. By shrinking $J_t$ if necessary, we can assume that $\gamma_{yz}(J_t) \subseteq I_1$. 
Visually, $\bigcup_{k=0}^n I_{s_k}$ covers $\beta_{t'}$ for all $t'$ in a neighbourhood of $t$ and each $\beta_{t'}$ ends in the top diamond $I_1$.

Now we let $t$ vary: doing the above described procedure for each $t \in [0,1]$, we end up with an open cover $\bigcup_t J_t$\newfoot{Similarly to $I_0$ and $I_1$ from above, the sets $J_0$ and $J_1$ contain $0$ and $1$ at the boundary, respectively. Visually, this means that at the edges of the original triangle, nearby geodesics can only be ``on one side'' of these edges.} of $[0,1]$. 
Again by a compactness argument, we can extract a finite subcover from $\bigcup_t J_t$, say $\bigcup_{l=0}^m J_{t_l} \supseteq [0,1]$. 
Order these set in the same way as the diamonds, i.e., increasing with respect to, say, the right endpoint, and remove unnecessary ones. Then the left endpoints are also ordered in an increasing fashion, subsequent sets overlap, and these are the only pairs of sets which overlap. 
In total, we obtain, modifying the above notation slightly, that
\begin{equation}
\bigcup_{l=0}^m \bigcup_{k=0}^{n_l} I_{s_k}^{t_l} \supseteq F([0,1],[0,1]),
\end{equation}
where $I_s^t$ is the diamond ``around'' $\beta_t(s)$, i.e., the $t$ emphasizes that this diamond belongs to the cover of $\beta_t$. 
The reward for this tedious construction is now the following: 
The triangle may be viewed as a fan consisting of $m$ pieces, and the covers of subsequent ``fan-geodesics'' share some geodesics between them. 
More precisely, as $J_{t_l} \cap J_{t_{l+1}} \neq \emptyset, l=0,\ldots, m-1$, and all geodesics ending in $\gamma_{yz}(J_{t_l})$ are contained in $\bigcup_{k=0}^{n_l} I_{s_k}^{t_l}$ (and the same for $l+1$), we know there exists some $\tilde{t}_l$ such that 
\begin{equation}
\beta_{\tilde{t}_l}([0,1]) \subseteq (\bigcup_{k=0}^{n_l} I_{s_k}^{t_l}) \cap (\bigcup_{k=0}^{n_{l+1}} I_{s_k}^{t_{l+1}}).
\end{equation}
A sketch of this process is depicted in Figure \ref{fig: tr covering}.
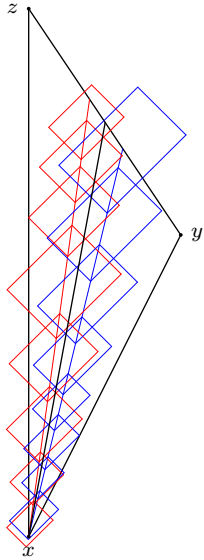
\begin{figure}
\begin{center}
\begin{tikzpicture}[line cap=round,line join=round,>=triangle 45,x=1cm,y=1cm]
\draw [line width=0.5pt] (0,0)-- (0,7);
\draw [line width=0.5pt] (0,0)-- (2,4);
\draw [line width=0.5pt] (2,4)-- (0,7);
\draw [line width=0.3pt,color=red] (0,0)-- (0.8004284450041361,5.799357332493796);
\draw [line width=0.3pt,color=blue] (0,0)-- (1.2376758244509984,5.143486263323502);
\draw [line width=0.3pt,color=red] (0.8267279802380822,5.989905786202704)-- (1.2533513419560691,5.563282424484717);
\draw [line width=0.3pt,color=red] (1.2533513419560691,5.563282424484717)-- (0.6901061231530742,5.000037205681722);
\draw [line width=0.3pt,color=red] (0.8267279802380822,5.989905786202704)-- (0.2634827614350872,5.4266605673997095);
\draw [line width=0.3pt,color=red] (0.2634827614350872,5.4266605673997095)-- (0.6901061231530742,5.000037205681722);
\draw [line width=0.3pt,color=red] (0.761762617575091,5.519211179246564)-- (1.2309613721927488,5.050012424628905);
\draw [line width=0.3pt,color=red] (1.2309613721927488,5.050012424628905)-- (0.6115064174933627,4.43055746992952);
\draw [line width=0.3pt,color=red] (0.6115064174933627,4.43055746992952)-- (0.14230766287570473,4.899756224547177);
\draw [line width=0.3pt,color=red] (0.14230766287570473,4.899756224547177)-- (0.761762617575091,5.519211179246564);
\draw [line width=0.3pt,color=red] (0.6773254792501515,4.907437396923022)-- (1.1917780245843772,4.392984851588796);
\draw [line width=0.3pt,color=red] (1.1917780245843772,4.392984851588796)-- (0.5125772061067496,3.713784033111169);
\draw [line width=0.3pt,color=red] (0.5125772061067496,3.713784033111169)-- (-0.0018753392274764025,4.228236578445395);
\draw [line width=0.3pt,color=red] (-0.0018753392274764025,4.228236578445395)-- (0.6773254792501515,4.907437396923022);
\draw [line width=0.3pt,color=red] (0.5696404732724808,4.127225458036948)-- (1.2163555248037299,3.4805104065056987);
\draw [line width=0.3pt,color=red] (1.2163555248037299,3.4805104065056987)-- (0.3625364577089074,2.626691339410876);
\draw [line width=0.3pt,color=red] (0.3625364577089074,2.626691339410876)-- (-0.284178593822342,3.2734063909421254);
\draw [line width=0.3pt,color=red] (-0.284178593822342,3.2734063909421254)-- (0.5696404732724808,4.127225458036948);
\draw [line width=0.3pt,color=red] (0.4081071475283464,2.9568654053135415)-- (0.9114482780555583,2.4535242747863295);
\draw [line width=0.3pt,color=red] (0.9114482780555583,2.4535242747863295)-- (0.24691719364730963,1.7889931903780807);
\draw [line width=0.3pt,color=red] (0.24691719364730963,1.7889931903780807)-- (-0.2564239368799024,2.292334320905293);
\draw [line width=0.3pt,color=red] (-0.2564239368799024,2.292334320905293)-- (0.4081071475283464,2.9568654053135415);
\draw [line width=0.3pt,color=red] (0.2745134046639475,1.9889364703892725)-- (0.7040901399591893,1.5593597350940307);
\draw [line width=0.3pt,color=red] (0.7040901399591893,1.5593597350940307)-- (0.1369457592818206,0.9922153544166621);
\draw [line width=0.3pt,color=red] (0.1369457592818206,0.9922153544166621)-- (-0.29263097601342114,1.4217920897119039);
\draw [line width=0.3pt,color=red] (-0.29263097601342114,1.4217920897119039)-- (0.2745134046639475,1.9889364703892725);
\draw [line width=0.3pt,color=red] (0.15475112854538114,1.121220888191584)-- (-0.2444346269675067,0.7220351326786961);
\draw [line width=0.3pt,color=red] (-0.2444346269675067,0.7220351326786961)-- (0.05792385435039467,0.4196766513607948);
\draw [line width=0.3pt,color=red] (0.05792385435039467,0.4196766513607948)-- (0.4571096098632824,0.8188624068736826);
\draw [line width=0.3pt,color=red] (0.4571096098632824,0.8188624068736826)-- (0.15475112854538114,1.121220888191584);
\draw [line width=0.3pt,color=red] (0.06519261675224118,0.47234113473392886)-- (-0.2875610229305885,0.11958749505109917);
\draw [line width=0.3pt,color=red] (-0.2875610229305885,0.11958749505109917)-- (-0.03257976709328441,-0.1353937607862049);
\draw [line width=0.3pt,color=red] (-0.03257976709328441,-0.1353937607862049)-- (0.3201738725895452,0.21735987889662475);
\draw [line width=0.3pt,color=red] (0.3201738725895452,0.21735987889662475)-- (0.06519261675224118,0.47234113473392886);
\draw [line width=0.3pt,color=blue] (0.14614555857785058,0.6073461710576342)-- (0.39270319510864926,0.36078853452683557);
\draw [line width=0.3pt,color=blue] (0.39270319510864926,0.36078853452683557)-- (-0.005110175142972295,-0.03702483572478599);
\draw [line width=0.3pt,color=blue] (-0.005110175142972295,-0.03702483572478599)-- (-0.25166781167377095,0.20953280080601266);
\draw [line width=0.3pt,color=blue] (-0.25166781167377095,0.20953280080601266)-- (0.14614555857785058,0.6073461710576342);
\draw [line width=0.3pt, color=blue] (1.4342352135642864,5.9603403198208165)-- (2.06882147188915,5.325754061495953);
\draw [line width=0.3pt, color=blue] (2.06882147188915,5.325754061495953)-- (1.0320589775438689,4.2889915671506715);
\draw [line width=0.3pt, color=blue] (1.0320589775438689,4.2889915671506715)-- (0.3974727192190053,4.923577825475535);
\draw [line width=0.3pt, color=blue] (0.3974727192190053,4.923577825475535)-- (1.4342352135642864,5.9603403198208165);
\draw [line width=0.3pt, color=blue] (1.1772051386362063,4.892184480031338)-- (1.7489399470299043,4.32044967163764);
\draw [line width=0.3pt, color=blue] (1.7489399470299043,4.32044967163764)-- (0.814861721606334,3.3863714462140693);
\draw [line width=0.3pt, color=blue] (0.814861721606334,3.3863714462140693)-- (0.2431269132126359,3.958106254607767);
\draw [line width=0.3pt, color=blue] (0.2431269132126359,3.958106254607767)-- (1.1772051386362063,4.892184480031338);
\draw [line width=0.3pt, color=blue] (0.9317702572328499,3.8722155058465506)-- (1.4302454512439144,3.373740311835486);
\draw [line width=0.3pt, color=blue] (1.4302454512439144,3.373740311835486)-- (0.6158559514202454,2.559350812011817);
\draw [line width=0.3pt, color=blue] (0.6158559514202454,2.559350812011817)-- (0.11738075740918097,3.0578260060228817);
\draw [line width=0.3pt, color=blue] (0.11738075740918097,3.0578260060228817)-- (0.9317702572328499,3.8722155058465506);
\draw [line width=0.3pt, color=blue] (0.7006635323699546,2.911790941343757)-- (1.0911512866940638,2.5213031870196474);
\draw [line width=0.3pt, color=blue] (1.0911512866940638,2.5213031870196474)-- (0.45318749079809323,1.883339391123677);
\draw [line width=0.3pt, color=blue] (0.45318749079809323,1.883339391123677)-- (0.06269973647398386,2.2738271454477865);
\draw [line width=0.3pt, color=blue] (0.06269973647398386,2.2738271454477865)-- (0.7006635323699546,2.911790941343757);
\draw [line width=0.3pt, color=blue] (0.519864224217965,2.160432031743459)-- (0.8065242144712333,1.8737720414901906);
\draw [line width=0.3pt, color=blue] (0.8065242144712333,1.8737720414901906)-- (0.33819020530359706,1.4054380323225542);
\draw [line width=0.3pt, color=blue] (0.33819020530359706,1.4054380323225542)-- (0.05153021505032851,1.6920980225758226);
\draw [line width=0.3pt, color=blue] (0.05153021505032851,1.6920980225758226)-- (0.519864224217965,2.160432031743459);
\draw [line width=0.3pt, color=blue] (0.3893583442542978,1.6180806440738975)-- (0.6665032973773406,1.3409356909508547);
\draw [line width=0.3pt, color=blue] (0.6665032973773406,1.3409356909508547)-- (0.2137145880007228,0.8881469815742369);
\draw [line width=0.3pt, color=blue] (0.2137145880007228,0.8881469815742369)-- (-0.06343036512232003,1.1652919346972797);
\draw [line width=0.3pt, color=blue] (-0.06343036512232003,1.1652919346972797)-- (0.3893583442542978,1.6180806440738975);
\draw [line width=0.3pt, color=blue] (0.2546485699922176,1.058258871874572)-- (0.46542001470496647,0.8474874271618231);
\draw [line width=0.3pt, color=blue] (0.46542001470496647,0.8474874271618231)-- (0.12106977722270248,0.5031371896795592);
\draw [line width=0.3pt, color=blue] (0.12106977722270248,0.5031371896795592)-- (-0.08970166749004632,0.713908634392308);
\draw [line width=0.3pt, color=blue] (-0.08970166749004632,0.713908634392308)-- (0.2546485699922176,1.058258871874572);
\draw [line width=0.5pt] (1.0062410789104006,5.490638381634399)-- (0,0);
\begin{scriptsize}
\coordinate [circle, fill=black, inner sep=0.5pt, label=270: {$x$}] (A1) at (0,0);
\coordinate [circle, fill=black, inner sep=0.5pt, label=180: {$z$}] (A1) at (0,7);
\coordinate [circle, fill=black, inner sep=0.5pt, label=0: {$y$}] (A1) at (2,4);
\end{scriptsize}
\end{tikzpicture}
\caption{$J_{t_l}$ ($t_l$ represented by the blue geodesic $\beta_{t_l}$) and $J_{t_{l+1}}$ ($t_{l+1}$ represented by the red geodesic $\beta_{t_{l+1}}$) overlap: $\tilde t_l \in J_{t_l} \cap J_{t_{l+1}}$ ($\tilde t_l$ represented by the black geodesic $\beta_{\tilde t_l}$).}
\label{fig: tr covering}
\end{center}
\end{figure}
We continue with the process of triangulation as follows: given $l$, consider the timelike triangle $\Delta(x, \beta_{\tilde{t}_l}(1), \beta_{t_{l+1}}(1))$. 
By construction, both $\beta_{\tilde{t}_l}$ and $\beta_{t_{l+1}}$ end in $\gamma_{yz}(J_{t_{l+1}}) \subseteq I_{s_{n_{l+1}}}^{t_{l+1}}$ and enter that set via $I_{s_{n_{l+1}}}^{t_{l+1}} \cap I_{s_{{n}_{l+1}-1}}^{t_{l+1}}$, i.e., they pass through the intersection of the ultimate and penultimate diamonds covering $\beta_{t_{l+1}}$. 
In particular, we can choose $\tilde{r}_1$ such that $\beta_{\tilde{t}_l}(\tilde{r}_1)$ is in said intersection. 
Note that the top governing point of the second to last diamond is timelike after the chosen point on $\beta_{\tilde{t}_l}$, i.e., 
\begin{equation}
\beta_{\tilde{t}_l}(\tilde{r}_1) \ll \beta_{t_{l+1}}(s_{n_{l+1}-1}^+). 
\end{equation}

By the openness of $\ll$, we can move a bit below the top governing point and still retain a timelike relation to the chosen point on $\beta_{\tilde{t}_l}$. In particular, both of these points are then contained in the intersection of the last two diamonds, and by the causal convexity also their connecting geodesic is entirely contained therein. 
More precisely, we find $r_1$ such that $\beta_{t_{l+1}}(r_1) \in I_{s_{n_{l+1}}}^{t_{l+1}} \cap I_{s_{{n}_{l+1}-1}}^{t_{l+1}}$ and $\beta_{\tilde{t}_l}(\tilde{r}_1) \ll \beta_{t_{l+1}}(r_1)$. 
Essentially, we constructed a quadrilateral consisting of 
$\beta_{\tilde{t}_l}(\tilde{r}_1), \beta_{t_{l+1}}(r_1), \beta_{\tilde{t}_l}(1)$ and  $\beta_{t_{l+1}}(1)$, which is completely contained in $I_{s_{n_{l+1}}}^{t_{l+1}}$. 
By transitivity of $\ll$, also the ``past most'' and ``future most'' points of this quadrilateral are timelike related, so we can split this into two timelike triangles 
$\Delta(\beta_{\tilde{t}_l}(\tilde{r}_1), \beta_{t_{l+1}}(r_1),\beta_{t_{l+1}}(1))$ and $\Delta(\beta_{\tilde{t}_l}(\tilde{r}_1), \beta_{\tilde{t}_l}(1), \beta_{t_{l+1}}(1))$, see Figure \ref{fig: tr subdivision}.
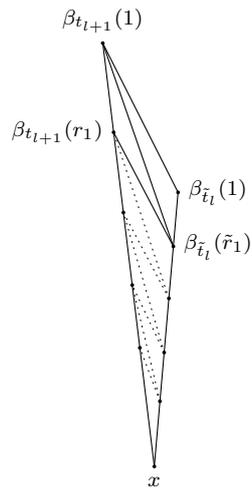
\begin{figure}
\begin{center}
\begin{tikzpicture}[line cap=round,line join=round,>=triangle 45,x=1cm,y=1cm]
\draw  (6,0.5)-- (6.312394086302661,4.134739071633176);
\draw  (6.312394086302661,4.134739071633176)-- (5.322453050287507,6.1132419316678295);
\draw  (5.322453050287507,6.1132419316678295)-- (6,0.5);
\draw (6.2507811308086705,3.4178656528584237)-- (5.322453050287507,6.1132419316678295);
\draw (6.2507811308086705,3.4178656528584237)-- (5.46509675524436,4.931488215129029);
\draw [dotted] (6.191614168646659,2.7294516317564024)-- (5.46509675524436,4.931488215129029);
\draw [dotted] (6.191614168646659,2.7294516317564024)-- (5.593471830172504,3.8679451590006977);
\draw [dotted] (6.129860772635827,2.01094417233817)-- (5.593471830172504,3.8679451590006977);
\draw [dotted] (6.129860772635827,2.01094417233817)-- (5.709622735837503,2.905675113571927);
\draw [dotted] (6.074288230040459,1.364351612690016)-- (5.709622735837503,2.905675113571927);
\draw [dotted] (6.074288230040459,1.364351612690016)-- (5.8100303344904045,2.073832916871458);
\begin{scriptsize}
\coordinate [circle, fill=black, inner sep=0.5pt, label=270: {$x$}] (A1) at (6,0.5);
\coordinate [circle, fill=black, inner sep=0.5pt, label=0: {$\beta_{\tilde{t}_l}(1)$}] (A1) at (6.312394086302661,4.134739071633176);
\coordinate [circle, fill=black, inner sep=0.5pt, label=90: {$\beta_{t_{l+1}}(1)$}] (A1) at (5.322453050287507,6.1132419316678295);
\coordinate [circle, fill=black, inner sep=0.5pt, label=0: {$\beta_{\tilde{t}_l}(\tilde{r}_1)$}] (A1) at (6.2507811308086705,3.4178656528584237);
\coordinate [circle, fill=black, inner sep=0.5pt, label=180: {$\beta_{t_{l+1}}(r_1)$}] (A1) at (5.46509675524436,4.931488215129029);
\coordinate [circle, fill=black, inner sep=0.5pt] (A1) at (6.191614168646659,2.7294516317564024);
\coordinate [circle, fill=black, inner sep=0.5pt] (A1) at (5.593471830172504,3.8679451590006977);
\coordinate [circle, fill=black, inner sep=0.5pt] (A1) at (6.129860772635827,2.01094417233817);
\coordinate [circle, fill=black, inner sep=0.5pt] (A1) at (5.709622735837503,2.905675113571927);
\coordinate [circle, fill=black, inner sep=0.5pt] (A1) at (6.074288230040459,1.364351612690016);
\coordinate [circle, fill=black, inner sep=0.5pt] (A1) at (5.8100303344904045,2.073832916871458);
\end{scriptsize}
\end{tikzpicture}
\caption{The process of subdividing a slim triangle.}
\label{fig: tr subdivision}
\end{center}
\end{figure}
Both of these triangles are entirely contained in the last timelike diamond $I_{s_{n_{l+1}}}^{t_{l+1}}$, so they satisfy the curvature bound by assumption. 
 As $\beta_{\tilde{t}_l}(\tilde{r}_1) \ll \beta_{t_{l+1}}(r_1)$  are also contained in $I_{s_{n_{l+1}-1}}^{t_{l+1}}$ , we can continue this procedure iteratively. After $n_{l+1}-1$ steps, we end up with $\beta_{\tilde{t}_l}(\tilde{r}_{n_{l+1}-1}) \ll \beta_{t_{l+1}}(r_{n_{l+1}-1})$ lying in $I_{s_{1}}^{t_{l+1}}$, where also $x$ lies. 

That is, in the end we have $n_{l+1}-1$ quadrilaterals, each of which we can split into two timelike triangles, and one additional timelike triangle at the bottom ending in $x$. Each of these timelike triangles is contained in one of the comparison neighbourhoods $\Set*{I_{s_{i}}^{t_{l+1}} \given i=1,\cdots,n_{l+1}}$ (i.e.\ the chain of comparison diamonds covering the geodesic $\beta_{t_{l+1}}$ and $\beta_{\tilde{t}_l}$), so satisfies the curvature bound, hence several applications of the Gluing Lemma \ref{gluinglemma} yield that the ``long and slim'' triangle $\Delta(x, \beta_{\tilde{t}_l}(1), \beta_{t_{l+1}}(1))$ satisfies the curvature bound. 
Note that in a triangle of the form $\Delta(x,\beta_{t_{l}}(1), \beta_{\tilde{t}_{l+1}}(1))$ 
the top side has a different time orientation from the point of view of the geodesic $\beta_{t_l}$ around which the covering is centred, but this changes nothing for the above described process.
So we can do this for all of the $2m-1$ long and slim triangles and apply the Gluing Lemma $2m-2$ times to obtain that the original triangle $\Delta(x,y,z)$ obeys the desired curvature bound.
\end{proof}

At first glance, our proof appears to be quite similar to the metric version by \citet*[Proposition II.4.9]{BH99}. There are however, some technical details we have to be wary of. In particular, in an arbitrary covering of the triangle, even if we use timelike diamonds, it is generally not true that we can achieve a subdivision consisting of timelike triangles such that each sub-triangle is contained within a comparison neighbourhood. We have to carefully construct the covering and then construct the sub-triangles in a seemingly complicated way as well. 
\medskip

As with the circle in the metric case (see Example \ref{ex:metricCircle}), it is possible to construct counterexamples to the automatic globalization of upper curvature bounds on Lorentzian pre-length spaces. The following is one such example, where the space has (local) curvature bounded above, but has neither unique nor continuously varying geodesics: 

\begin{ex}[The Lorentzian cylinder]
We set the Lorentzian cylinder to be the spacetime $X=\mb{R}\times \mathbb{S}^1$, i.e., take a strip $\mb{R}\times[0,2\pi]$ in Minkowski space and glue the boundary as depicted by the arrows in Figure \ref{fig: lor circle}(which is not to be confused with the totally vicious cylinder $\mathbb{S}^1 \times \mathbb{R}$!). This space is locally isometric to Minkowski space, hence clearly has (local) timelike curvature bounded above by 0.

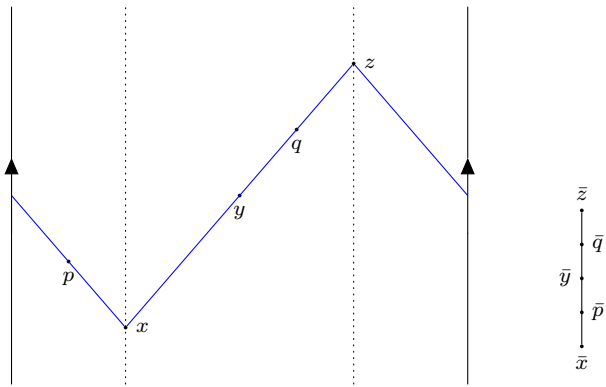
\begin{figure}[ht]
\begin{center}
\begin{tikzpicture}[line cap=round,line join=round,>=triangle 45,x=1cm,y=1cm]
\draw [dotted] (0,-0.5)-- (0,4.5);

\draw [->] (-1.5,1.5) -- (-1.5,2.5);
\draw (-1.5,2.5) -- (-1.5,4.5);
\draw (-1.5,2.5) -- (-1.5,-0.5);

\draw [dotted] (3,-0.5) -- (3,4.5);
\draw [->] (4.5,1.5) -- (4.5,2.5);
\draw (4.5,1.5) -- (4.5,-0.5);
\draw (4.5,2.5) -- (4.5,4.5);
\begin{scriptsize}
\coordinate [circle, fill=black, inner sep=0.5pt, label=0: {$x$}] (x) at (0,0.25);
\coordinate [circle, fill=black, inner sep=0.5pt, label=0: {$z$}] (z) at (3,3.75);
\coordinate [circle, fill=black, inner sep=0.5pt, label=270: {$y$}] (y) at (1.5,2);
\coordinate [circle, fill=black, inner sep=0.5pt, label=270: {$p$}] (p) at (-0.75,1.125);
\coordinate [circle, fill=black, inner sep=0.5pt, label=270: {$q$}] (q) at (2.25,2.875);
\coordinate [circle, fill=black, inner sep=0.5pt, label=270: {$\bar{x}$}] (bx) at (6,0);
\coordinate [circle, fill=black, inner sep=0.5pt, label=90: {$\bar{z}$}] (bz) at (6,{sqrt(3.5^2-3^2)});
\coordinate [circle, fill=black, inner sep=0.5pt, label=180: {$\bar{y}$}] (by) at (6,{sqrt(3.5^2-3^2)/2});
\coordinate [circle, fill=black, inner sep=0.5pt, label=0: {$\bar{p}$}] (bp) at (6,{sqrt(3.5^2-3^2)/4});
\coordinate [circle, fill=black, inner sep=0.5pt, label=0: {$\bar{q}$}] (bq) at (6,{3*sqrt(3.5^2-3^2)/4});
\end{scriptsize}
\draw[color=blue] (x) -- (z);
\draw[color=blue] (x) -- (-1.5,2);
\draw[color=blue] (4.5,2) -- (z);
\draw (bx) -- (bz);
\end{tikzpicture}
\caption{The Lorentzian cylinder. The depicted triangle fails to satisfy an upper curvature bound since its comparison triangle is degenerate. }
\label{fig: lor circle}
\end{center}
\end{figure}

Take two (dotted) vertical lines on the cylinder, which are directly opposite each other (as in Figure \ref{fig: lor circle}). Given two timelike related points, with one point on each of the two vertical lines, there exist precisely two geodesics between these points (wrapping around to the right and to the left on the cylinder, respectively). 
Consider any such pair of points, denoted $x$ and $z$, and the corresponding pair of geodesics, and choose as a third vertex some point $y$ on one of the two geodesics. Then the comparison triangle for $\Delta(x,y,z)$ is clearly degenerate. 
Choosing two points $p$ and $q$ on the two different geodesics (and at different parameters, say $p$ occurs at an earlier parameter than $q$) sufficiently far away from the endpoints of the two geodesics, $p$ and $q$ will not be timelike related, i.e., $\tau(p,q)=0$. However, as the comparison triangle is essentially a line segment and the two points are at different parameters, we clearly have $\tau(\bar{p}, \bar{q}) >0$, violating global upper curvature bounds by $0$. 
\end{ex}

However, as in the metric case, some of our additional assumptions are relatively mild and can be recovered from the global curvature bounds. More specifically, as in Proposition \ref{pop: Elementary Catk} for CAT($k$) spaces, we may show that \LpLSs with global curvature bounded above by $K \in \mb{R}$ automatically possess unique geodesics between points with $\tau$-distance less than $D_K$.

\begin{thm}[Unique geodesics in upper curvature bounds]\label{thm: Lor unique geodesics}
Let $X$ be a strongly causal and regular Lorentzian pre-length space with timelike curvature bounded above by $K \in \mathbb{R}$. Let $x\ll y$ be in a comparison neighbourhood $U \subseteq X$ and suppose $\tau(x,y)<D_K$. Then there exists a unique geodesic from $x$ to $y$ contained in $U$. In particular, if $X$ satisfies a global upper curvature bound, geodesics between timelike related points in $X$ with $\tau$-distance less than $D_K$ are unique. 
\end{thm}

\begin{proof} 
Suppose towards a contradiction that there is another geodesic from $x$ to $y$. Denote the curves corresponding to two of these geodesic segments by $\alpha_1$ and $\alpha_2$, respectively, and parametrize them with constant speed on $[0,1]$. 
Let $p \in \alpha_1((0,1))$ and consider the timelike triangle $\Delta(x,p,y)$, which satisfies size-bounds for $\lm{K}$. Clearly the comparison triangle in $\lm{K}$ is degenerate. As $\alpha_1 \neq \alpha_2$, there exists $t \in (0,1)$ such that $\alpha_1(t) \neq \alpha_2(t)$. 
Then there exist neighbourhoods $V_1$ and $V_2$ of $\alpha_1(t)$ and $\alpha_2(t)$, respectively, such that $V_1 \cap V_2 = \emptyset$.
As $X$ is strongly causal, we find points $x_1,y_1, \ldots, x_n,y_n$ and $p_1,q_1, \ldots, p_m,q_m$ such that $\alpha_1(t) \in U_1:=\bigcap_{i=1}^n I_1(x_i^1,y_i^1) \subseteq V_1$ and $\alpha_2(t) \in U_2 := \bigcap_{j=1}^m I_1(p_j^2,q_j^2) \subseteq V_2$.
By the continuity of $\alpha_1$, there is some neighbourhood $I_1$ of $t$ such that $\alpha_1(I_1) \subseteq U_1$. As $U_1$ is causally convex by definition, all diamonds with endpoints inside $U_1$ are contained in $U_1$. Similarly for $U_2$.
In particular, there exists $\varepsilon >0$ such that $D_1:=I(\alpha_1(t-\varepsilon),\alpha_1(t+\varepsilon)) \subseteq U_1$ and $D_2:=I(\alpha_2(t-\varepsilon),\alpha_2(t+\varepsilon)) \subseteq U_2$.
Then $\alpha_2(t) \notin D_1$ (and vice versa), so 
either $\alpha_1(t-\varepsilon) \not\ll \alpha_2(t)$ or $\alpha_2(t) \not\ll \alpha_1(t+\varepsilon)$. However, in $\lm{K}$ we have $\widebar{\alpha}_1(s)=\widebar{\alpha}_2(s)$ for all $s \in [0,1]$ . In particular, $\widebar{\alpha}_1(t-\varepsilon) \ll \widebar{\alpha}_2(t) \ll \widebar{\alpha}_1(t+\varepsilon)$, a contradiction to upper timelike curvature bounds.
\end{proof}

The astute reader may recall that, in the metric setting, we also automatically have that geodesics between points with distance less than $\diam{(M_k)}$ vary continuously with their endpoints in spaces with global upper curvature bounds, cf.\ Proposition \ref{pop: Elementary Catk}. 
We may do something similar in the synthetic Lorentzian setting under the additional assumption that diamonds between endpoints of $\tau$-distance less than $D_K$ are compact. 
This is, of course, closely related to global hyperbolicity, or rather, depending on the sign of $K$, \emph{global hyperbolicity of order $\sqrt{-K}$} as introduced in \citet*{Har82}. 
In particular, any uniquely geodesic \LLS that is globally hyperbolic (of order $\sqrt{-K}$ if $K<0$) satisfies the assumptions of the following proposition. 

\begin{pop}[Continuous variation of geodesics]
\label{pop: continuous variation of geodesics}
Let $K\in\mb{R}$ and let $X$ be a regular, strongly causal, non-totally imprisoning, locally causally closed and non-timelike locally isolating \LpLS satisfying (i) and (ii) from Definition \ref{TLCB}. 
Suppose that causal diamonds between points of $\tau$-distance less than $D_K$ are compact and that such points possess a unique geodesic connecting them. 
Then $G|_{\tau^{-1}((0,D_K))\times[0,1]}$ is continuous.
\end{pop}
\begin{proof}
Let $x\ll y\in X$ with $\tau(x,y)<D_K$. Let $x_n\to x$ and $y_n\to y$, then we can without loss of generality assume that $x_n\ll y_n$. 
By non-timelike local isolation, openness of $\tau^{-1}([0,D_K))$, and continuity of $\tau|_{\tau^{-1}([0,D_K))}$, there are points $x_-\ll x\ll y\ll y_+$ with $\tau(x_-,y_+)<D_K$, and we can again without loss of generality assume $x_n,y_n$ all lie in the compact set $J(x_-,y_+)\subseteq\tau^{-1}([0,D_K))$. 

As $X$ is non-totally imprisoning, there exist uniform Lipschitz reparamet\-rizations of $\gamma_{x_ny_n}$ on a common domain, cf.\ \citet*[Remark 1.6.28]{Ber20}. Now we use the limit curve theorem (on compact bounded domains) of \citet*[Theorem~3.7]{KS18}, to get a subsequence converging uniformly to a limit curve $\beta$ which connects $x$ to $y$. 

In order to obtain upper semicontinuity of the length functional on $J(x_-,y_+)$ (see \citet*[Proposition~3.17]{KS18}), continuity of $\tau$ on $J(x_-,y_+)$ is sufficient. With this in mind, we follow the proof of \citet*[Theorem~2.23]{BORS23} for a fixed $T$ to obtain that the limit curve $\beta$ is a geodesic from $x$ to $y$.

Since geodesics between points at $\tau$-distance less than $D_K$ are unique, $\beta$ is a reparametrization of $\gamma_{xy}$. As $\tau$ is continuous, $\gamma_{x_ny_n}$ also converges to $\gamma_{xy}$ in the constant speed parametrization on $[0,1]$, proving the claim.

\end{proof}

\begin{rem}[Globalization of continuity]
If we strengthen the requirements from Theorem \ref{thm: Lor AlexPatch} on $X$ and the geodesic map, instead imposing that geodesics between any $x\ll y$ exist and are unique, and that the geodesic map $G$ on its full domain ${\ll} \times [0,1]$ is continuous, it follows that $\tau$ is continuous at any pair of timelike related points. and even finite by \citet*[Lemma 2.25]{KS18}). 
This fact can be viewed as a converse statement to the previous Proposition \ref{pop: continuous variation of geodesics}. 

Indeed, as $X$ has timelike curvature bounded above in the sense of Definition \ref{TLCB}, it follows that $\tau|_{U\times U}$ is continuous for any $U$ from the covering of comparison neighbourhoods constructed in Lemma \ref{lem: autoSizeBounds}.\newfoot{In essence, we pass from local curvature bounds in the sense of Definition \ref{TLCB} to those in the sense of \citet*[Definition 4.7]{KS18}, mimicking localizability and allowing the application of \citet*[Proposition 3.17]{KS18}.} 
By using that $X$ is geodesic (between timelike related points) and slightly adapting the proof of \citet*[Proposition 3.17]{KS18}, it can be shown that $L_{\tau}$ is upper semi-continuous. Combining this with the continuity of $G$, we find $L_{\tau}(G(x,y,\cdot))=L_{\tau}(\gamma_{xy})=\tau(x,y)$, hence $\tau$ is both upper and lower semi-continuous and is therefore continuous on $\ll$. 
In particular, if $\tau$ is continuous on its full domain, then $X$ automatically satisfies condition \emph{(i)} of Definition \ref{TLCB} and we need not assume that it does so in Theorem \ref{thm: Lor AlexPatch}. 
\end{rem}

We conclude this subsection with the following corollary which should be of interest to the smooth spacetime community. The result can be regarded as an analogue to \citet*[Theorem 3.4]{Har82}, which provides a globalization theorem for smooth spacetimes with smooth timelike (sectional) curvature bounded above by $K \in \mathbb{R}$. Here, we instead treat the case where the spacetime has smooth timelike curvature bounded below by $K$. 
As conventions in the synthetic literature were chosen to align with \citet*{AB08}, there is a mismatch between the curvature bounds as discussed in \citet*{Har82} and our work. More precisely, upper curvature bounds in \citet*{Har82} correspond to lower curvature bounds in our setting. See \citet*{BKOR24} for more details.

\begin{cor}[Spacetime globalization]
Let $K \in \mb{R}$ and let $M$ be a smooth spacetime which is globally hyperbolic (of order $\sqrt{-K}$ if $K<0$). 
Suppose that for all $(x,y) \in \tau^{-1}((0,D_K))$ there is a unique (timelike) geodesic joining them. 
Further suppose that $M$ has smooth timelike curvature bounded below by $K$ in the sense of \citet*{BKOR24}, i.e. the sectional curvature of all timelike planes at all points $p \in M$ is bounded below by $K$. 
Then $M$ has global timelike curvature bounded above by $K$ in the sense of Definition \ref{TLCB}. 
\end{cor}
\begin{proof}
A smooth timelike curvature bound from below by $K$ is equivalent to a (synthetic) timelike curvature bound from above by $K$ in the sense of Definition \ref{TLCB}, according to \citet*[Theorems 3.1 and 3.2]{BKOR24}. 
The continuity of the geodesic map $G$ follows by Proposition \ref{pop: continuous variation of geodesics}, hence $M$ satisfies all assumptions to apply Theorem \ref{thm: alex patch}. 
\end{proof}

\subsection{Timelike curvature bounded below}
\label{subsec: TLCBB}

Finally, we show that a bound may be placed on the (finite) diameter (see Definition \ref{def: fin diam}) of a Lorentzian pre-length space with negative lower timelike curvature bound. This result is in the spirit of the Bonnet--Myers Theorem (Theorem \ref{thm: meyers}) from Riemannian geometry However, as for previous results, we need to be careful about the Lorentzian subtleties. 
\bigskip

First, due to the way timelike curvature bounds were introduced by \citet*{KS18},  the hierarchy of curvature bound implications is reversed when compared to the metric setting. 
This was done to keep consistency with the notation introduced in \citet*{AB08}. More precisely, recall that if a metric space has curvature bounded below by some $k$, then it also has curvature bounded below by all $k' \leq k$. Similarly, if it has $k$ as an upper curvature bound, it also has any $k' \geq k$ as an upper curvature bound. 
For timelike curvature bounds in the Lorentzian setting, the inequalities are reversed and any \LpLS possessing timelike curvature bounded below by $K$ does so for all $K' \geq K$, while any \LpLS with timelike curvature bounded above by $K$ does so for all $K' \leq K$ (see \citet*[Lemma 3.27]{BS22} for a proof of this result and \citet*[Example 4.9]{KS18} for an explanation for this choice of convention). Hence, we shall be required to assume a \emph{negative} lower bound, as opposed to a positive one (cf.\ Theorem \ref{thm: meyers}). In addition, due to the behaviour of Anti-de Sitter discussed before Definition \ref{def: fin diam}, we consider the finite diameter of Lorentzian pre-length spaces, rather than the ordinary diameter as in the metric case.
\medskip

Before diving into the theorem proper, we provide a technical lemma which yields a non-degeneracy condition for adjacent timelike sub-triangles. 

\begin{lem}[Non-degeneracy condition] \label{lem: non-deg-cond}
Let $X$ be a strongly causal, locally causally closed, regular, and geodesic \LpLS and let $U$ be a comparison neighbourhood in $X$. Let $a\ll z$ in $U$ and let $\alpha$ be a geodesic in $U$ starting at $a$ and ending at $z$. Let $x=\alpha(t)$ and let $y\in I(x,z)$. Assume that both $\Delta(a,x,y)$ and $\Delta(x,y,z)$ satisfy size bounds. Let $\beta$ be a timelike geodesic starting at $x$ and ending at $y$, and denote by $\alpha_-=\alpha|_{[0,t]}$ and $\alpha_+=\alpha|_{[t,1]}$ the parts of $\alpha$ in the past and future of $x$, respectively. 

\begin{enumerate}
\item If $X$ has timelike curvature bounded below by $K$ and $\Delta(x,y,z)$ is non-degenerate, then $\Delta(a,x,y)$ is also non-degenerate, and $\ma_x(\alpha_+,\beta)$ and $\ma_x(\alpha_-,\beta)$ are equal and positive. 
\item If $X$ has timelike curvature bounded above by $K$ and the angle \linebreak $\ma_x(\alpha_-,\beta)$ exists and $\Delta(a,x,y)$ is non-degenerate, then $\Delta(x,y,z)$ and (if it satisfies size bounds) $\Delta(a,y,z)$ are also non-degenerate, and both $\ma_x(\alpha_\pm,\beta)>0$ though they need not be equal.
\end{enumerate}
\end{lem}
\begin{proof}
\emph{(i)} As $\Delta(x,y,z)$ is non-degenerate, also the corresponding comparison triangle is non-degenerate and hence $\tilde{\ma}_x(y,z)>0$. By angle comparison, cf.\ Theorem \ref{thm: equivalent curv bounds}.\emph{(i)}, we get 
$0<\tilde{\ma}_x(y,z) \leq \ma_x(\alpha_+, \beta)$. As $X$ is locally causally closed, strongly causal and has timelike curvature bounded below, we can apply Proposition \ref{pop: equal angles along geodesic}, from which it follows that $\ma_x(\alpha_+,\beta)=\ma_x(\alpha_-,\beta)>0$, and again by angle comparison, we have 
$\tilde{\ma}_x(a,y)\geq\ma_x(\alpha_-,\beta)>0$. In particular, also 
$\Delta(a,x,y)$ is non-degenerate. 

\emph{(ii)} For the curvature bounded above case, the arguments of the curvature bounded below case reverse (we only get $\tilde{\ma}_x(a,y) \leq \ma_x(\alpha_-, \beta)$ from the triangle inequality of angles, see \citet*[Theorem 4.5.(i)]{BS22}), and monotonicity comparison at the angle at $a$ yields the statement on the big triangle.
\end{proof}

The result of Theorem \ref{thm: lor meyers} should be closely compared to \citet*[Proposition 5.10]{CM20}. In this pioneering work, the authors introduce synthetic Ricci curvature bounds using optimal transport methods, hence their result might be even closer in spirit to the original Bonnet--Myers Theorem than the one shown below. 
See also \citet*{BC23} for the corresponding result in the setting of spacetimes with low regularity. 
Moreover, should it prove true that Ricci curvature bounds (using optimal transport) are weaker than sectional curvature bounds (using triangle comparison) in the Lorentzian picture, as is the case for metric curvature comparison, cf. \citet*{Pet19}, then our result has narrower scope. However, as the hierarchy of curvature bounds is not yet known and our method is distinct, the proof of Theorem \ref{thm: lor meyers} is valuable in its own right.
\medskip

In the statement of Theorem \ref{thm: lor meyers}, we impose an additional non-degeneracy condition on the space $X$: for each pair of points $x \ll z$ in $X$, there exists a $y \in X$ such that $\Delta(x,y,z)$ is a non-degenerate timelike triangle. This allows us to apply Lemma \ref{lem: non-deg-cond} along timelike geodesics and ensures that the space is not locally one-dimensional; compare this to the class of one-dimensional spaces which the Bonnet--Myers theorem does not apply to in the metric setting, as discussed before Theorem \ref{thm: Toponogov}. Upon extending the following result to \LpLSs with local curvature bounded below by $K<0$, it may be possible to classify such one-dimensional spaces by following the approach of \citet*[Proposition 8.44]{AKP19} and this could be explored in future work.

\begin{thm}[Bound on the finite diameter]
\label{thm: lor meyers}
Let $X$ be a strongly causal, locally causally closed, regular, and geodesic\newfoot{Global curvature bounds guarantee the existence of a geodesic for all $a\ll b$ with $\tau(a,b)<D_K$. In this context, however, we will need the existence of geodesics for all timelike related pairs of points slightly further apart. $X$ being geodesic is a sufficient condition for this.} Lorentzian pre-length space which has global curvature bounded below by $K<0$. Assume that, for each pair of points $x\ll z$ in $X$, there exists $y \in X$ such that $\Delta(x,y,z)$ is a non-degenerate timelike triangle.
Then $\diam_{\mathrm{fin}}(X)\leq D_K$.
\end{thm}
\begin{proof}
Without loss of generality, we only consider $K=-1$. 
Let indirectly $a,b\in X$ with $\tau(a,b)=\pi+\varepsilon$ for some small enough $\varepsilon>0$, and let $\alpha: [0, \pi + \varepsilon] \to X$ be a timelike geodesic from $a$ to $b$ parameterized by $\tau$-arclength. 
Let 
$x=\alpha(t_-)$ and $z=\alpha(t_+)$
for $t_-=\frac{\pi}{2}+\frac{\varepsilon}{2}$ (i.e., the midpoint), $t_+=\frac{\pi}{2}+\frac{\pi}{8}$. 
Note that the specific value of $t_+$ is not important, any $t_+ \in (t_-,\pi)$ suffices. The corresponding point $z$ is mainly used further on in the proof to ensure that a triangle with longest side shorter than $\tau(a,z) < \pi$, is realizable in $\lm{-1}$.
Denote by $\alpha_-=\alpha|_{[0,t_-]}$ and $\alpha_+=\alpha|_{[t_-,\pi + \varepsilon]}$ the parts of $\alpha$ in the past and future of $x$, respectively.

By the non-degeneracy assumption on $X$, we find a point 
$y \in I(x,z)$
such that 
$\tau(x,z) > \tau(x,y) + \tau(y,z)$, i.e., $\Delta(x,y,z)$ is non-degenerate.
Let $\beta$ be a geodesic from $x$ to $y$.
By Lemma \ref{lem: non-deg-cond}.\emph{(i)}, we get that  
$\Delta(a,x,y)$ is non-degenerate and $\omega \coloneqq \ma_x(\alpha_-,\beta)=\ma_x(\alpha_+,\beta)>0$. 
We now claim that $\tau(a,b) < \tau(a,y) + \tau(y,b)$, contradicting the reverse triangle inequality. 

We name the lengths: 
$t_-=\tau(a,x)=\tau(x,b)\eqqcolon t$, $\tau(a,y) \eqqcolon p$, $\tau(y,b)\eqqcolon q$ and $\tau(x,y) \eqqcolon m$, so the claim reads 
\begin{equation}
    \label{eq: claim bonnet}
    2t<p+q. 
\end{equation}
 We create a situation in $\lm{K}$ consisting of comparison hinges for $(\alpha_-,\beta)$ and $(\beta,\alpha_+)$, giving the triangles $\Delta(\tilde a,\tilde x,\tilde y)$ and $\Delta(\tilde x,\tilde y,\tilde b)$, see Figure \ref{fig: bonnet comp sit2}. Note that these triangles are non-degenerate since $\omega >0$. 
We name the side-lengths $\tau(\tilde a,\tilde y)=\tilde p$, $\tau(\tilde y,\tilde b)=\tilde q$. 

\begin{figure}[ht]
\begin{center}
\begin{tikzpicture}
\begin{scriptsize}
\draw (-4,0) .. controls (-4.5,1) and (-3.7,3) .. (-4,4) node (A)[circle, fill=black,inner sep=0.5pt,pos=0.4,label=left:$x$]{};
\end{scriptsize}
\draw (A) .. controls (-3.5,3.3) and (-3.1, 3.2) .. (-3,3.2);
\draw (-3,3.2) .. controls (-3.1,3.5) and (-3.2, 3.6) .. (-4,4);
\draw (-4,0) .. controls (-3.1,1) and (-3.3,2.5) .. (-3,3.2);

\coordinate (y1) at (-3.455,3.2);
\coordinate (a1) at (-4.23,0);
\coordinate (b1) at (-3.849,4);

\draw (1,0) -- (1,4);
\draw[dashed] (1.5,3.4) -- (1,4);
\draw (1,1.7) -- (1.5,3.4);
\draw [dashed] (1.5,3.4) -- (1,0);

\begin{scriptsize}
\coordinate [circle, fill=black, inner sep=0.5pt, label=270: {$a$}] (a) at (-4,0);
\coordinate [circle, fill=black, inner sep=0.5pt, label=90: {$b$}] (b) at (-4,4);
\coordinate [circle, fill=black, inner sep=0.5pt, label=0: {$y$}] (y) at (-3,3.2);
\coordinate [] (y0) at (A);

\coordinate [circle, fill=black, inner sep=0.5pt, label=270: {$\tilde a$}] (ba) at (1,0);
\coordinate [circle, fill=black, inner sep=0.5pt, label=90: {$\tilde b$}] (bb) at (1,4);
\coordinate [circle, fill=black, inner sep=0.5pt, label=180: {$\tilde x$}] (bx) at (1,1.7);
\coordinate [circle, fill=black, inner sep=0.5pt, label=0: {$\tilde y$}] (by) at (1.5,3.4);

\pic [draw, ->, "${\omega}$", angle radius=9mm, angle eccentricity=1.5] {angle = by--bx--bb};
\pic [draw, ->, "${\omega}$", angle eccentricity=0.7] {angle = ba--bx--by};

\pic [draw, ->, "${\omega}$", angle radius=9mm, angle eccentricity=1.5] {angle = y1--A--b1};
\pic [draw, ->, "${\omega}$", angle eccentricity=0.7] {angle = a1--A--y1};
\end{scriptsize}
\end{tikzpicture}
\caption{The construction in $X$ and comparison hinges.}
\label{fig: bonnet comp sit2}
\end{center}
\end{figure}
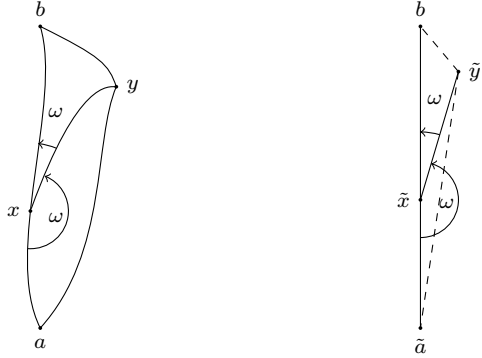

By hinge comparison, cf.\ Theorem \ref{thm: equivalent curv bounds}.(ii), we get 
$p=\tau(a,x) \geq \tau(\tilde{a},\tilde{x})=\tilde{p}$ and $q=\tau(y,b) \geq \tau(\tilde{y},\tilde{b})=\tilde{q}$.
We claim that $2t<\tilde{p}+\tilde{q}$. As $p\geq \tilde{p}$ and $q\geq \tilde{q}$, this implies the above claim.

By the reverse triangle inequality, we have $t > m+\tilde{q}$ and $\tilde{p} > t+m > 2m+\tilde{q}$, thus $\tilde{p}-\tilde{q} > 2m$. 
Note that the reverse triangle inequality yields strict inequalities since the triangles are non-degenerate.
Recall that $\tilde p \leq p = \tau(a,y) \leq \tau(a,z) = \frac{\pi}{2}+\frac{\pi}{8}$ and since $\tilde{q} \geq 0$, it follows that
\begin{equation}
0 < 2m < \tilde{p} - \tilde{q} \leq \frac{\pi}{2} + \frac{\pi}{8} < \pi\,.
\end{equation}
In particular, as cosine is decreasing, we have 

\begin{equation}
    \label{eq: cos inequality}
    0<\cos(\frac{\tilde{p}-\tilde{q}}{2})< \cos(m). 
\end{equation}

We now write down the equations for $\omega=\ma_{\tilde{x}}(\tilde{a},\tilde{y})=\ma_{\tilde{x}}(\tilde{y},\tilde{b})$ in the law of cosines (cf.\ \citet*[Lemma 2.4]{BS22} and remember $K=-1$):

\begin{align*}
\cos(m)\cos(t)-\sin(m)\sin(t)\cosh(\omega)&=\cos(\tilde{p}),\\
\cos(m)\cos(t)+\sin(m)\sin(t)\cosh(\omega)&=\cos(\tilde{q}).
\end{align*}
We add these two equations to eliminate $\omega$ and then use the cosine addition formula:
\begin{equation}
\cos(m)\cos(t) = \frac{1}{2}( \cos(\tilde{p})+\cos(\tilde{q}))=\cos(\frac{\tilde{p}+\tilde{q}}{2})\cos(\frac{\tilde{p}-\tilde{q}}{2}) .
\end{equation}

Rearranging further, we find
\begin{equation}
\cos(t)\frac{\cos(m)}{\cos(\frac{\tilde{p}-\tilde{q}}{2})}=\cos(\frac{\tilde{p}+\tilde{q}}{2}).
\end{equation}
As $0<\cos(\frac{\tilde{p}-\tilde{q}}{2})< \cos(m)$, we know that the fraction on the left hand side is bigger than one, and since $\frac{\pi}{2}<t=\tau(a,x)=\tau(x,b)<\pi$, we have $\cos(t)<0$. In total, we get
\begin{equation}
\cos(\frac{\tilde{p}+\tilde{q}}{2})<\cos(t),
\end{equation}
and as $\cos$ is monotonically decreasing, we obtain 
\begin{equation}
\tilde{p}+\tilde{q}>2t,
\end{equation}
which by above arguments implies the original claim \eqref{eq: claim bonnet}.
\end{proof}

\begin{rem}[Global hyperbolicity and spacetimes]
If we additionally assume that $X$ is a globally hyperbolic Lorentzian length space, then the result can be extended to apply to the diameter, in addition to the finite diameter. Indeed, $\tau$ is then finite, cf.\ \citet*[Theorem 3.28]{KS18}, so the diameter and the finite diameter agree. Our result may then be viewed as an extension of \citet*[Theorem 9.5]{BE79}, in which a bound is derived for the diameter of a globally hyperbolic spacetime with timelike sectional curvature bounded below by $K<0$. 
\end{rem}

\begin{rem}[An implication of Theorem \ref{thm: lor meyers}]
There is an immediate corollary to the metric version of Theorem \ref{thm: meyers}, which states that the perimeter of any triangle in a space with curvature bounded below by $k$ cannot be greater than $\frac{2\pi}{\sqrt{k}}$, see \citet*[Corollary~10.4.2]{BBI01}. This can be argued using hinge comparison. Typically, deriving Lorentzian analogues of metric results is at least as difficult as the original metric derivation, so it is noteworthy that this corollary is easier to show in the Lorentzian world. In fact, the result immediately follows from the reverse triangle inequality: let $\Delta(x,y,z)$ be a timelike triangle, then by Theorem \ref{thm: lor meyers}, we know $D_K > \tau(x,z) \geq \tau(x,y) + \tau(y,z)$, hence $\tau(x,y) + \tau(y,z) + \tau(x,z) < 2 D_K$ as required. 
\end{rem}

\section{Outlook}\label{sec:concOutlook}

Finally, let us discuss potential future research stemming from this paper. 
In particular, we wish to highlight the existence of a Lorentzian version of the famous Toponogov Globalization Theorem (Theorem \ref{thm: Toponogov}) for lower curvature bounds. 
In the smooth Lorentzian case, this was achieved by \citet*{Har82}. 
A corresponding result in the synthetic Lorentzian setting is presented in the upcoming work by \citet*{BHNR23}, where curvature bounds in the sense of angle comparison and the cat's cradle method of \citet*{LS13} are used. 
The assumptions used by \citet*[Theorem~3.6]{BHNR23} are not especially mild in comparison to the metric setting, where \citet*{BGP92} prove the Toponogov Globalisation Theorem for complete spaces and \citet*{Pet16} does so for non-complete geodesic spaces. 
In \citet*[Section~3.4]{BGP92}, the Toponogov Globalization Theorem for complete length spaces is formulated in terms of the so-called four point condition (see \citet*[Proposition 10.1.1]{BBI01}). An analogous condition has been presented for (lower) curvature bounds in Lorentzian pre-length spaces by \citet*[Definition 4.6]{BKR23}, which might yield an alternative way of obtaining the Lorentzian version of the Toponogov Theorem, with milder assumptions. Regardless, the four point condition will likely be a strong tool in the arsenal of synthetic Lorentzian geometry. 
\bigskip

One immediate benefit of obtaining a Lorentzian Toponogov Globalization Theorem is the extension of Theorem \ref{thm: lor meyers} to spaces with local timelike curvature bounded below by $K<0$ in the spirit of the metric Bonnet--Myers Theorem \ref{thm: meyers}, as opposed to only those with global bounds. A generalization in this direction is discussed in detail in \citet*[Theorem~4.3]{BHNR23} and utilizes the Lorentzian Toponogov Theorem found therein. 
A rigidity statement for this theorem has been developed in \citet*{Ber24}, where the assumption of a geodesic of maximal length forces the space to behave like a warped product. 
As for the Toponogov-style globalization theorem, it may be possible to relax the assumptions required to obtain a bound on the finite diameter of a Lorentzian pre-length space with local timelike curvature bounded below by $K<0$, either by using the four point condition or by taking a more direct approach.
\bigskip

Another area of interest is the study of polyhedral spaces, in particular those which describe graphs. By viewing graphs as metric spaces, we may impose curvature bounds on them in order to obtain information about their structure. For example, \citet*[Section 4.2]{BBI01} note that all locally finite, connected graphs are non-positively curved, that is, they have curvature bounded above by $k\leq 0$. Furthermore, \citet*[Theorems II.5.4, II.5.5]{BH99} demonstrate that a graph $X$ has local/global curvature bounded above by $k$ if and only if $X$ is locally/globally uniquely geodesic up to a distance $\diam(M_k)$, providing a drastic simplification of the Globalization Theorem \ref{thm: alex patch} in this setting. It follows that if a graph has global curvature bounded above by $k \leq 0$, then the graph is a tree. Furthermore, \citet*[Chapter II.5]{BH99} show that a weighted graph is CAT($k$) if and only if every locally injective loop in the graph has length greater than $\diam(M_k)$, from which it follows that weighted trees are CAT($k$) for any $k$. For more results in this vein, see \citet*{Dav24}. 
Turning to the synthetic Lorentzian setting, note that any graph which possesses vertices of valency $>2$ exhibits branching curves given by the edges adjacent to said vertices. Consequently, graphs consisting of vertices and edges cannot be made a \LpLSn, as they violate the lower semi-continuity of the time separation function, cf.\ \citet*[Lemma 2.12]{KS18}. Since the edges are the obstacle here, the natural choice is to exclude them and instead only consider the point-cloud given by the vertices, namely a \emph{causal set}.
\medskip

Causal sets are sets equipped with locally finite partial order, which are used in the theory of quantum gravity to model discrete spacetimes, see the work of Bombelli, Dowker, Henson, Lee, Meyer, and Sorkin \cite{BLMS87,DHS04}. They may also be represented by locally finite, transitively reduced, directed, acyclic graphs called Hasse diagrams, where the additional qualifiers reflect the properties of the inferred Lorentzian geometry. 
However, the edges of these diagrams are merely indicators of the partial ordering and should not be considered part of the set itself. 
\citet*[Section 5.3]{KS18} show that, similarly to how graphs can be given metric space structure, causal sets can be given the structure of a Lorentzian pre-length space, though not a Lorentzian length space. Given the interplay between upper curvature bounds and graph structure, it is expected that global upper curvature bounds would impose significant restrictions on the structure of causal sets. 
It is also expected that causal sets possess non-negative curvature (perhaps under a constraint on the finiteness of $\tau$, in the spirit of connectedness). 
\bigskip

Another open question, initially posed by \citet*[Section 5.3]{KS18}, asks if it is possible to discretize Lorentzian length spaces, and by extension spacetimes, using causal sets. In metric geometry, it is known that compact length spaces are given by the Gromov--Hausdorff limit of finite graphs, see \citet*[Proposition 7.5.5]{BBI01}, which amounts to a discretization of a continuous space, so it is not too far-fetched to presume that a similar result would hold in Lorentzian signature. Indeed, recall that a \emph{bounded Lorentzian metric space} is a set $X$ equipped with a function $\tau : X\times X \rightarrow [0,\infty)$ which satisfies the reverse triangle inequality, distinguishes points, and is continuous in a topology where $\Set*{ (p,q) \given \tau(p,q) \geq \epsilon }$ is compact for all $\epsilon >0$. \citet*[Proposition 2.3, Corollary 4.32]{MS23} show that the finite bounded Lorentzian metric spaces are precisely the causal sets and every (not necessarily finite) bounded Lorentzian metric space is the limit of causal sets. Furthermore, later work by \citet*[Lemma 4.1]{BHNR23} shows that causal diamonds (modulo their spacelike boundaries), in globally hyperbolic, regular Lorentzian length spaces, are bounded Lorentzian metric spaces. It follows that a globally hyperbolic, regular Lorentzian length space can be described locally by the limit of causal sets, though it is not clear whether this statement is true globally. In particular, gluing the local causal sets together may not result in a causal set; we need to ensure that the result is still locally finite and acyclic.\newfoot{Recall that a partial order is antisymmetric and transitive.}
\medskip

Conversely, it is known that not all causal sets have continuum approximation given by spacetimes, see the excellent review paper by \citet*{Sur19} for more details. However, it is plausible that a wider range of causal sets have continuum approximation given by sequences of (non-finite, non-discrete) Lorentzian pre-length spaces. The Morse spacetimes of Borde and Sorkin are another candidate class of spaces whose limit could reasonably yield a causal set, since they admit singular structures such as causal discontinuities and causal bubbling, which otherwise arise due to topological changes when taking the limits of spacetimes. See \citet*{Gar23} and \citet*{DJ24} for recent results concerning the structure of Morse spacetimes. 
Strengthening these correspondences would drive the development of the Lorentzian pre-length space, causal set, and Morse spacetime frameworks forward in parallel, enabling consolidation and verification of the results derived in each picture.
\medskip

Of course, Gromov--Hausdorff convergence of spaces is an interesting topic in its own right, in particular due to the stability of curvature bounds under Gromov--Hausdorff limits in the metric setting, see \citet*{Kap02}. A schema for an analogous Lorentzian result can be found in \citet*[Section 4.1]{BHNR23} and convergence of warped product Lorentzian length spaces was studied in \citet*{KS21}, though the precise statement is still an open problem.

\addcontentsline{toc}{section}{References}
\bibliography{references}

\end{document}